\documentclass[11pt]{article}
\usepackage[dvips]{epsfig}
\usepackage{amsmath,amssymb,euscript,graphicx,amsfonts,verbatim}
\usepackage{enumerate,color}
\usepackage{graphicx}
\let\bowtie\relax
\DeclareFontEncoding{LS1}{}{}
\DeclareFontSubstitution{LS1}{stix}{m}{n}
\DeclareSymbolFont{STIXsymbols}{LS1}{stixscr}{m}{n}
\DeclareMathSymbol{\bowtie}{\mathrel}{STIXsymbols}{"0E}

\usepackage{graphicx}

\usepackage[colorlinks=true,linkcolor=blue,urlcolor=blue,citecolor=blue]{hyperref}
\usepackage{float}

\newtheorem{thm}{Theorem}[section]
\newtheorem{prop}[thm]{Proposition}
\newtheorem{lem}[thm]{Lemma}
\newtheorem{cor}[thm]{Corollary}
\newtheorem{qst}[thm]{Question}

\newtheorem{defn}[thm]{Definition}
\newtheorem{prob}[thm]{Problem}

\newtheorem{rmk}[thm]{Remark}

\newcommand{\Qed}{\rule{2.5mm}{3mm}}

\newcommand{\irr}[1]{\operatorname{Irr}(#1)}

\newcommand{\RR}{\mathbb{R}}
\newcommand{\CC}{\mathbb{C}}

\DeclareMathOperator\sym{Sym}
\DeclareMathOperator\alt{Alt}
\DeclareMathOperator{\fix}{fix}

\newcommand{\triv}[1]{\operatorname{\mathbf{1}}_{#1}}
\newcommand{\ind}[2]{\triv{#1}^{#2}}
\newcommand{\sgn}[1]{\mathsf{sgn}(#1)}
\newcommand{\stab}[2]{\operatorname{Stab}(#1,#2)}

\newcounter{case}

\renewcommand{\thecase}{\arabic{case}}

\newcounter{subcase}

\numberwithin{subcase}{case}

\newenvironment{proof}{{\noindent \sc Proof.}}{\hfill $\Qed$ \\}
\begin{document}


\begin{center}
{\bf\Large On cocliques in commutative Schurian association schemes of the symmetric group} \\ [+4ex]  
\large Roghayeh Maleki{} and  
\addtocounter{footnote}{0} 
Andriaherimanana Sarobidy Razafimahatratra{\footnote{ corresponding author}} 
\\ [+2ex]
{\it \small 
University of Primorska, UP FAMNIT, Glagolja\v ska 8, 6000 Koper, Slovenia\\
University of Primorska, UP IAM, Muzejski trg 2, 6000 Koper, Slovenia\\
}
\end{center}

\begin{abstract}
	Given the symmetric group $G = \sym(n)$ and a multiplicity-free subgroup $H\leq G$, the orbitals  of the action of $G$ on $G/H$ by left multiplication induce  a commutative association scheme. The irreducible constituents of the permutation character of $G$ acting on $G/H$ are indexed by partitions of $n$ and if $\lambda \vdash n$ is the second largest partition in dominance ordering among these, then the Young subgroup $\sym(\lambda)$ admits two orbits in its action on $G/H$, which are $\mathcal{S}_\lambda$ and its complement. 
	
	In their monograph [Erd\H{o}s-Ko-Rado theorems: Algebraic Approaches. {\it Cambridge University Press}, 2016] (Problem~16.13.1), Godsil and Meagher asked whether $\mathcal{S}_\lambda$ is a coclique of a graph in the commutative association scheme arising from the action of $G$ on $G/H$. If such a graph exists, then they also asked whether its smallest eigenvalue is afforded by the $\lambda$-module.
	
	In this paper, we initiate the study of this question by taking $\lambda = [n-1,1]$.
	We show that the answer to this question is affirmative for the pair of groups $\left(G,H\right)$, where $G = \sym(2k+1)$ and $H = \sym(2) \wr \sym(k)$, or $G = \sym(n)$ and $H$ is one of $\alt(k) \times \sym(n-k),\ \alt(k) \times \alt(n-k)$, or $\left(\alt(k)\times \alt(n-k)\right) \cap \alt(n)$. For the pair $(G,H) = \left(\sym(2k),\sym(k)\wr \sym(2)\right)$, we also prove that the answer to this question of Godsil and Meagher is negative.
\end{abstract}

\begin{quotation}
\noindent{\em Math. Subj. Class.:} 05C50, 20C30, 05C35.
\end{quotation}


\section{Introduction}

Let $X$ be a finite non-empty set and let $\mathcal{R} = \{ R_0,R_1,\ldots,R_k \}$ be a collection of relations on $X$. We say that the pair $(X,\mathcal{R})$ is an \emph{association scheme} if the following statements are satisfied:
\begin{enumerate}[(i)]
	\item $\{ (x,x) :\ x \in X \} \in \mathcal{R}$,\label{first}
	\item $\mathcal{R}$ is a partition of $X \times X$,\label{second}
	\item for any $i\in \{0,1,\ldots,k\}$, the relation $\{(y,x) :\ (x,y) \in R_i\}$ belongs to $\mathcal{R}$,\label{third}
	\item for any $(x,y)\in R_k$, the number $|\left\{ z\in X :\ (x,z) \in R_i \mbox{ and }(z,y)\in R_j \right\}|$ is a number $p_{ij}^k$ depending only on $i,j$ and $k$ and not the choice of $x$ and $y$.\label{fourth}
\end{enumerate}
The size of $X$ is the \emph{order} of the association scheme $(X,\mathcal{R})$ and the number of relations $k+1$ is its \emph{rank}. It is worth noting that a non-empty set $X$ and a relation $R$ on $X$ determines a digraph whose vertex set is $X$ and for any $x,y\in V(X)$, an arc between $x$ and $y$ occurs if and only if $(x,y)\in R$. Consequently, each relation in an association scheme determines a digraph on $X$ and therefore an adjacency matrix. An association scheme $(X,\mathcal{R})$ is \emph{symmetric} if the adjacency matrices corresponding to the relations are symmetric. Moreover, $(X,\mathcal{R})$ is \emph{commutative} if the corresponding adjacency matrices commute with each other. A survey on commutative association schemes can be found in \cite{martin2009commutative}.

An example of well-known association schemes is the \emph{triangular association scheme} $(X,\mathcal{R})$, where $X =  \left\{ A \subset \{1,2,\ldots,n\} :\ |A| = 2 \right\}$ and $\mathcal{R} = \{ R_0,R_1,R_2 \}$, with $R_i = \{ (A,B) \in X \times X :\ |A\cap B| = 2-i \}$, for $i \in \{0,1,2\}$.

Association schemes were introduced in the 50s by Bose and Shimamoto \cite{bose1952classification}, and the study of these objects has developed into a major area of study in algebraic combinatorics since then. A generalization of these objects known as \emph{coherent configurations} were also introduced by Higman \cite{higman1970coherent} in the 70s to study permutation groups. 
Another well-known example of association schemes arises from the action of a finite transitive group $G\leq \sym(\Omega)$, where $\Omega$ is a finite non-empty set. If $\mathcal{O}$ is the set of all orbitals of the action of $G$ on $\Omega$, i.e., its orbits in the induced action on $\Omega \times \Omega$, then $(\Omega,\mathcal{O})$ is an association scheme. 

The association scheme $(\Omega,\mathcal{O})$ is called the \emph{orbital scheme} of $G$. In addition, an association scheme arising as an orbital scheme is called \emph{Schurian}. The (di)-graphs in an orbital scheme are called orbital (di)-graphs. An example of Schurian association schemes is again the triangular association scheme which arises from the action of $\sym(n)$ on the $2$-subsets of $\{1,2,\ldots,n\}$.

Next, we recall a result about the commutativity of Schurian association schemes. Let $H$ be a subgroup of a group $G$. We will denote the trivial character of $H$ by $\triv{H}$ and the induced character of $\triv{H}$ on the group $G$ by $\ind{H}{G}$. Let $\irr{G} = \{ \phi_1,\phi_2,\ldots,\phi_t \}$ be a complete set of (complex) irreducible characters of $G$. We say that $H$ is a \emph{multiplicity-free} subgroup of $G$ if the irreducible decomposition
\begin{align*}
	\mathbf{1}_H^G = \sum_{i=1}^t m_i \phi_i, 
\end{align*}
is such that $m_i \in \{0,1\}$, for $i\in \{0,1,\ldots,k\}$. If $H\leq G$ is multiplicity-free, then we say that the pair $(G,H)$ is a \emph{Gelfand pair}. Gelfand pairs are well studied and we refer the readers to \cite[Chapter~4]{ceccherini2008harmonic} for details on them. A consequence of $(G,H)$ being a Gelfand pair is that the Bose-Mesner algebra of the Schurian association scheme induced by the action of $G$ on $G/H$ coincides with its Hecke algebra (also known as the double centralizer algebra).
The commutativity of Schurian association schemes is determined by the point stabilizer of the corresponding transitive groups. Let $G\leq \sym(\Omega)$ be a transitive group and let $H$ be the point stabilizer of $G$. The association scheme which arises from the transitive group $G$ is commutative if and only if $(G,H)$ is a Gelfand pair (see \cite[Chapter~13]{godsil2016erdos} for details).

As we will study commutative Schurian association schemes arising from transitive actions of the symmetric group, we need to review some facts on the representation theory of these groups. Recall that the irreducible submodules $\mathbb{C}\sym(n)$ are indexed by partitions of the integer $n$. For any partition $\lambda \vdash n$, the irreducible $\mathbb{C}\sym(n)$-module corresponding to $\lambda$ is the \emph{Specht module} $S^\lambda$. Therefore, the irreducible characters of $\sym(n)$ are also indexed by the partitions of the integer $n$. Given a partition $\lambda \vdash n$, we will denote the corresponding irreducible character by $\chi^\lambda$. For any partition $\lambda = [\lambda_1,\lambda_2,\ldots,\lambda_k]$ of $n$, the \emph{Young subgroup} $\sym(\lambda)$ is the subgroup $\sym(\lambda_1) \times \sym(\lambda_2) \times \ldots \times \sym(\lambda_k)$. For any two partitions $\lambda = [\lambda_1,\lambda_2,\ldots,\lambda_k]$ and $\mu = [\mu_1,\mu_2,\ldots,\mu_t]$ of $n$, we say that $\lambda$ \emph{dominates} $\mu$ and write $\mu \trianglelefteq \lambda$, if $\sum_{i = 1}^j \lambda_i \geq \sum_{i = 1}^j \mu_i$, for all $j\in \{1,2,\ldots,t\}$. 

Now, we will describe the problem considered in this paper. Henceforth, we assume that $G = \sym(n)$ and $H$ is a multiplicity-free subgroup of $G$. Let $$\Lambda(n,H) := \left\{ \lambda \vdash n :\ \langle \ind{H}{G},\chi^\lambda \rangle = 1 \right\}.$$ The eigenvalues of the orbital digraphs corresponding to $(G,H)$ are indexed by the partitions in $\Lambda(n,H)$. In particular, an eigenspace of any orbital digraph is a direct sum of certain irreducible $\mathbb{C}\sym(n)$-modules, whose corresponding partitions appear in $\Lambda(n,H)$. For any $\lambda \in \Lambda(n,H)$, we refer to the subspace given by the Specht module $S^\lambda$ as the $\lambda$-module. The partition $[n]$ is always an element of $\Lambda(n,H)$ due to the fact that $G$ acts transitively on $G/H$ by left multiplication. Let $\lambda \neq [n]$ be the partition which is the second largest in dominance ordering in $\Lambda(n,H)$. By \cite[Theorem~13.9.1]{godsil2016erdos},  the Young subgroup $\sym(\lambda)$ admits two orbits in its action on $G/H$. In \cite[Problem~16.13.1]{godsil2016erdos}, Godsil and Meagher asked the following question.
\begin{qst}[Godsil-Meagher]
	Let $G = \sym(n)$ and $H$ be a multiplicity-free subgroup of $G$, and define $\Omega = G/H$. Assume that $\lambda\vdash n$ is the second largest in dominance ordering in $\Lambda(n,H)$ and $\{ \mathcal{S},\Omega\setminus \mathcal{S} \}$ is the orbit partition of $\sym(\lambda)$ on $\Omega$.
	\begin{enumerate}[(a)]
		\item Is there an orbital graph of $G$ acting on $G/H$ in which $S$ is a coclique?\label{qst:part1}
		\item Is the eigenvalue corresponding to the $\lambda$-module the least eigenvalue for such graphs?\label{qst:part2}
	\end{enumerate}\label{qst:main}
\end{qst}

\subsection{Motivation}
Our motivation to study this question is closely related  to a famous extremal set theory theorem. The Erd\H{o}s-Ko-Rado (EKR) theorem is one of the most important results in extremal combinatorics. We say that a collection of $k$-subsets of $[n]:= \{1,2,\ldots,n\}$ is a \emph{$t$-intersecting family} if $|A\cap B|\geq t$ for all $A,B\in \mathcal{F}$. The EKR theorem is stated as follows.
\begin{thm}[EKR \cite{erdos1961intersection}]
	For any two positive integers $k\geq t$, there exists $n_0(k,t)$ such that if $n\geq n_0(k,t)$ and $\mathcal{F}$ is a $t$-intersecting family of $k$-subsets of $[n]:=\{1,2,\ldots,n\}$, then $|\mathcal{F}|\leq \binom{n-t}{k-t}$. In addition, equality holds if and only if there exists $S\subset [n]$ of size $t$ such that
	\begin{align*}
		\mathcal{F} = \{ A \subset [n] :\ |A|=k \mbox{ and } S\subset A \}.
	\end{align*}\label{thm:EKR}
\end{thm}
For the case where $t = 1$, Erd\H{o}s, Ko and Rado also proved that $n_0(k,1) = 2k+1$.

The EKR theorem has been extensively studied and extended to other objects such as vector spaces \cite{katona1972simple} and permutations \cite{Frankl1977maximum}. There are various proofs of the EKR theorem which range from purely combinatorial \cite{erdos1961intersection}, to probabilistic \cite{katona1972simple} and algebraic \cite{wilson1984exact}. 

We now exhibit the relations between a certain association scheme and the EKR theorem. Let $n\geq 2k$ be two positive integers and define $\binom{[n]}{k}$ to be the collection of all $k$-subsets of $[n]$. For any $0\leq i\leq k$, define 
\begin{align*}
	\mathcal{O}_i = \left\{ (A,B) \in \binom{[n]}{k} \times \binom{[n]}{k} :\ |A\cap B| = k-i \right\}.
\end{align*}
The \emph{Johnson scheme} $\mathcal{J}(n,k)$ is the association scheme given by $\left( \binom{[n]}{k},\{ \mathcal{O}_0,\mathcal{O}_1,\ldots,\mathcal{O}_k \} \right)$. The Johnson scheme contains the \emph{Johnson graph} $J(n,k)$ and the \emph{Kneser graph} $K(n,k)$, which are the orbital graphs of $\mathcal{O}_1$ and $\mathcal{O}_k$, respectively. The Kneser graph is important in the study of the EKR theorem since it encodes the $1$-intersecting sets of $k$-subsets of $[n]$. Given a \emph{coclique} (or independent set) $\mathcal{S}$ of $K(n,k)$, it is not hard to check that $\mathcal{S}$ has the property that $A\cap B \neq \varnothing$ for all $A,B \in \mathcal{S}$, i.e., it is $1$-intersecting. Conversely, any collection of $k$-subsets of $[n]$ with the property that any two elements intersect is a coclique of $K(n,k)$. Hence, a collection $\mathcal{F}$ of $k$-subsets of $[n]$ is a maximum $1$-intersecting family if and only if it is a maximum coclique in the Kneser graph $K(n,k)$. Using the Johnson scheme, Wilson \cite{wilson1984exact} also gave an algebraic proof of the EKR theorem relying on the well-known Hoffman bound (see Theorem~\ref{thm:ratiobound}) and described precisely the smallest bound $n_0(k,t)$ in Theorem~\ref{thm:EKR} for which the results of the EKR theorem hold. 

One important aspect of the Johnson scheme is that it is Schurian, that is, it arises from the orbital scheme of a transitive group. The corresponding group action for the case of the Johnson scheme is $\sym(n)$ acting on the $k$-subsets of $[n]$, or equivalently, on the cosets of $\sym(k) \times \sym(n-k)$. Another important property of $\sym(k) \times \sym(n-k)$ is that it is a multiplicity-free subgroup of $\sym(n)$. Hence, the Johnson scheme $\mathcal{J}(n,k)$ is a commutative association scheme. 
The second largest partition in dominance ordering in the corresponding permutation character is $[n-1,1]$. Further, one of the two orbits of the Young subgroup $\sym([n-1,1])$ is a maximum coclique in the Kneser graph $K(n,k)$ and its smallest eigenvalue is afforded by the $[n-1,1]$-module. In other words, Question~\ref{qst:main} is true for the Gelfand pair $(G,H) = (\sym(n),\sym(k) \times \sym(n-k))$. 

Now, let $(G=\sym(n),H)$ be a Gelfand pair for which $\lambda \in \Lambda(n,H)$ is the second largest in dominance ordering.  If Question~\ref{qst:main} is true for $(G,H)$, then we obtain an EKR type theorem on $G/H$ or on the corresponding combinatorial objects as follows. Since Question~\ref{qst:main}~\eqref{qst:part1} is true, there exists an orbital graph in which an orbit of $\sym(\lambda)$ acting on the cosets $G/H$ is a coclique. Let $\mathcal{K}(G,H)$ be the union of all orbital graphs with this property. 
Let us consider the following terminologies.
\begin{defn}\hfil
	\begin{enumerate}[(1)]
		\item  Two cosets $xH$ and $yH$ of $H$ are $(G,H)$-intersecting if they are not adjacent in $\mathcal{K}(G,H)$.
		\item A collection $\mathcal{F}$ of cosets in $G/H$ is $(G,H)$-intersecting if any pair of its elements are $(G,H)$-intersecting.
		\item The orbit of a conjugate of $\sym(\lambda)$ on $G/H$ which is a $(G,H)$-intersecting family is called a $(G,H)$-canonical intersecting family.
	\end{enumerate}
\end{defn}

{As Question~\ref{qst:main}~\eqref{qst:part2} is also true, one can prove that any maximum $(G,H)$-intersecting family (i.e., a maximum coclique in $\mathcal{K}(G,H)$) has size equal to the size of the $(G,H)$-intersecting orbit of $\sym(\lambda)$. Using these terminologies, one can pose the following question.}
\begin{prob}
	Assume that $(G,H)$ is a Gelfand pair for which Question~\ref{qst:main} is true. Are the $(G,H)$-canonical intersecting families the only $(G,H)$-intersecting families of maximum size?\label{prob:main}
\end{prob}

If the answer to Problem~\ref{prob:main} is true, then we obtain a full EKR theorem for the cosets $G/H$ in the sense that the $(G,H)$-intersecting families have size at most the size of a $(G,H)$-canonical intersecting family, and those attaining this bound must be a $(G,H)$-canonical intersecting family.

\subsection{Main results}

For the case where $G = \sym(2k)$ and $H = \sym(2)\wr \sym(k)$, the association scheme corresponding to the Gelfand pair $(G,H)$ is the \emph{perfect matching scheme} of the complete graph $K_{2k}$ (see \cite{godsil2016erdos} for details on this).  An EKR type theorem on perfect matchings of the complete graph $K_{2n}$ was proved by Godsil and Meagher in \cite{godsil2016algebraic}, and also Lindzey in \cite{lindzey2017erdHos}. The main technique used in the proof of this EKR type theorem can be extended to prove that the answer to Question~\ref{qst:main} for the Gelfand pair $(\sym(2k),\sym(2)\wr \sym(k))$ is affirmative.

The answer to Question~\ref{qst:main} is however not always affirmative. The Gelfand pair $(G,H) = (\sym(8),\sym(4)\wr \sym(2))$ induces a commutative rank $3$ association scheme. The character table of this scheme is
\begin{align*}
	\left(
	\begin{tabular}{ccc|c}
		$1$ & $16$ & $18$ &  $1$\\
		$1$ & $-4$ & $3$ & $14$\\
		$1$ & $2$ & $-3$ & $20$\\
	\end{tabular}\right).
\end{align*}
Using the Ratio Bound (see Theorem~\ref{thm:ratiobound}), it is not hard to see that a coclique in either of the two non-trivial orbital graphs is of size at most $7$. Moreover, since $\Lambda(8,H) = \{ [8],[6,2],[4,4] \}$, the second largest in dominance ordering is $[6,2]$ which induces two orbits of size $15$ and $20$. Consequently, we have a negative answer to Question~\ref{qst:main}. 
Our first result, which is a generalization of this example, is stated as follows.
\begin{thm}
	The answer to Question~\ref{qst:main} is negative for the multiplicity-free subgroup $\sym(k)\wr \sym(2)$ of $\sym(2k)$.\label{thm:main1}
\end{thm}

\renewcommand{\arraystretch}{1.25} 

\begin{table}[t]
	{\small
		\begin{center}
		\begin{tabular}{lllccc}
			Line & Group & & $n$ & Index & Rank\\
			1&$\sym(k) \times \sym(n-k)$ & $(2k \leq n)$ & $n$ & ${n \choose k}$ & $k+1$ \\
			2&$\alt(k) \times \sym(n-k)$ & $(2k \leq n$, $k \neq 2)$ & $n$ & $2{n \choose k}$ &  $k+3$\\
			3&$\sym(k) \times \alt(n-k)$ & $(2k \leq n$, $k \neq n-2)$ & $n$ & $2{n \choose k}$ &  $k+3$\\
			4& $\alt(k) \times \alt(n-k)$ & $(k\geq 3, 2k\leq n-2)$ &$n$ & $4\binom{n}{k}$ & $2k+6$\\
			5& $\left(\sym(k)\times \sym(n-k)\right)\cap \alt(n)$ & $(2k\leq n, (n,k)\neq (4,2))$ & $n$ & $2\binom{n}{k}$ & $2k+2$\\
			6 & $\left(\sym(2)\wr \sym(k)\right) \times \sym(1)$ & &$2k+1$ & & \\
		\end{tabular}
		\caption{Multiplicity-free intransitive groups \label{table:intrans2}}
	\end{center}}
\end{table}

In this paper, we initiate the study of Question~\ref{qst:main} by considering the Gelfand pair $(\sym(n),H)$ for which the second largest partition in dominance ordering is $[n-1,1]$. The multiplicity-free subgroups of $\sym(n)$ were classified in \cite{godsil2010multiplicity}. We will focus on the families of multiplicity-free subgroups of $\sym(n)$ given in Table~\ref{table:intrans2} in this work. 

We state our next main result.
\begin{thm}
	The answer to Question~\ref{qst:main} is affirmative for the Gelfand pairs $(\sym(n),H)$, where $H$ is the group in lines~2-6 of Table~\ref{table:intrans2}.\label{thm:main2}
\end{thm}

\subsection{Structure of the paper}
This paper is organized as follows. In the Section~\ref{sect:background} and Section~\ref{sect:association-scheme}, we give some background results on permutation groups, spectral graph theory techniques and association schemes. In Section~\ref{sect:proof1}, we give a proof of Theorem~\ref{thm:main1}. In  Section~\ref{sect:line1}, we review the proof that Question~\ref{qst:main} is true for $(\sym(n),\sym(k)\times \sym(n-k))$. The proof of Theorem~\ref{thm:main2} is spread in Section~\ref{sect:line2-3}, Section~\ref{sect:line4}, Section~\ref{sect:line5}, and Section~\ref{sect:qpm}. We conclude this paper by stating some interesting questions and problems in Section~\ref{sect:conclusion}.

\section{Background results}\label{sect:background}
\subsection{Permutation groups}
Let $G$ be a group and $H\leq G$ be a subgroup. The group $G$ acts on $G/H$ through $g (xH) = gxH$, for any $g\in G$ and $xH\in G/H$. It is not hard to see that this action is transitive since for any $xH,yH \in G/H$, we have $(yx^{-1})(xH) = yH$. We will denote the stabilizer of $xH$ in this action of $G$ on $G/H$ by $\stab{G}{xH}$. It is clear that $\stab{G}{H} = H$. It is well known that any finite transitive group $G\leq \sym(\Omega)$ is permutation equivalent to $G$ in its action on $G/H$, where $H$ is the stabilizer of any $\omega \in \Omega$. This correspondence enables us to switch between the cosets or a set of specific combinatorial objects, whichever is easier. For any $\omega \in \Omega$, we let $\stab{G}{\omega}$ be the stabilizer of $\omega$ in $G$.

The transitive group $G$ acting on $G/H$ induces an action on $G/H \times G/H$ by componentwise multiplication. The orbits of this action are called the \emph{orbitals} of $G$ and the \emph{rank} of $G$ is the number of orbitals. By transitivity of $G$, the set $\{ (xH,xH) :\ xH \in G/H \}$ is always an orbital of $G$. Other orbitals of $G$ must partition the set $\{ (xH,yH) :\ xH \neq yH \in G/H \}$. If $\mathcal{O}$ is an orbital of $G$ such that $\mathcal{O} = \left\{ (yH,xH):\ (xH,yH) \in \mathcal{O} \right\}$, then we say that $\mathcal{O}$ is \emph{self-paired}, otherwise, it is called \emph{non-self-paired}. 

The next lemma gives a well-known correspondence between orbitals and \emph{suborbits} (i.e., orbits of a point stabilizer). Its proof can be found in any standard textbook on permutation groups.
\begin{lem}
	The map which takes any orbital $\mathcal{O}$ to the suborbit $\{xH:\ (H,xH) \in \mathcal{O}\}$ is a bijection.
\end{lem}

For any $g \in G$, define
\begin{align}
	\fix_H^G(g) = \left|\left\{ xH \in G/H :\ g \left(xH\right) = xH \right\}\right|.
\end{align}

The \emph{permutation character} of the action of $G$ on $G/H$ is the character $\fix_H^G : G \to \CC$ which takes any $g\in G$ to $\fix_H^G(g)$.
The next lemma is a well-known result in the theory of permutation groups.
\begin{lem}[Orbit Counting lemma]
	The number of orbits of $G$ acting on $G/H$ is 
	\begin{align*}
		\frac{1}{|G|} \sum_{g\in G} \fix_H^G(g).
	\end{align*}\label{lem:orbit-counting-lemma}
\end{lem}

Since the action of $G$ on $\Omega$ induces an action on $\Omega \times \Omega$, the number of orbitals of $G$ (i.e., the orbits on $\Omega \times \Omega$) can be easily computed using Lemma~\ref{lem:orbit-counting-lemma}.
\begin{cor}
	The number of orbitals of $G$ acting on $G/H$ is 
	\begin{align*}
		\frac{1}{|G|} \sum_{g\in G} \fix_H^G(g)^2.
	\end{align*}\label{cor:number-of-orbitals}
\end{cor}

Let us now write these two formulae in terms of the natural inner product on the space of characters. Let $G$ be a finite group. We will denote the complete set of complex irreducible characters of $G$ by $\operatorname{Irr}(G)$. Recall that given $\phi,\psi \in \operatorname{Irr}(G)$, a natural inner product between $\phi$ and $\psi$ is given by
\begin{align*}
	\langle \phi,\psi \rangle = \frac{1}{|G|} \sum_{g\in G} \phi(g) \overline{\psi(g)}.
\end{align*}

It is not hard to see that, in fact, $\fix_H^G = \ind{H}{G}$. Using Lemma~\ref{lem:orbit-counting-lemma}, it is clear that the trivial character $\mathbf{1}_G$ of $G$ is always a constituent of $\ind{H}{G}$. Moreover, its multiplicity is equal to $1$ since
\begin{align*}
	\left\langle \ind{H}{G},\triv{G} \right\rangle = \frac{1}{|G|} \sum_{g\in G} \fix(g) = 1.
\end{align*}
 Using Corollary~\ref{cor:number-of-orbitals}, it is easy to see that the rank of the transitive group $G$ (or the corresponding association scheme) is $\left\langle\ind{H}{G},\ind{H}{G} \right\rangle$.

\subsection{Equitable partitions}
Let $X= (V,E)$ be a graph. A partition $\pi = \{ V_1,V_2,\ldots,V_k \}$ of the vertex set of $X$ is called \emph{equitable} if for any $1\leq i,j\leq k$ there exists a number $a_{ij}$ such that the number of neighbours in $V_j$ of any vertex of $ V_i$ is equal to $a_{ij}$. If $\pi$ is an equitable partition of $X$, then the \emph{quotient graph } $X/\pi$ is the directed multigraph with vertex set equal to the elements of $\pi$ and for any $i,j\in \{1,2,\ldots,k\}$, there are exactly $a_{ij}$ arcs from $V_i$ to $V_j$. The \emph{quotient matrix} of an equitable partition $\pi$ is the $k\times k$ matrix $A(X/\pi)$ indexed in its rows and columns by $\{1,2,\ldots,k\}$ and whose $ij$-entry is $a_{ij}$.

Equitable partitions are important in spectral graph theory due to the following result.
\begin{thm}\cite[Lemma~2.2.2]{godsil2016erdos}
	If $\pi$ is an equitable partition of $X$, then the characteristic polynomial of $A(X/\pi)$ divides the characteristic polynomial of the adjacency matrix of $X$. In particular, an eigenvalue of $A(X/\pi)$ is an eigenvalue of $X$.
\end{thm}\label{thm:eigenvalues-equitable-partitions}

The above theorem enables one to determine certain eigenvalues of the graph $X$ through $X/\pi$. Given an eigenvalue of $A(X/\pi)$ and a corresponding eigenvector $v$, one can lift $v$ into an eigenvector of $X$. For any equitable partition $\pi = \{ V_1,V_2,\ldots,V_k \}$, define the $|V(X)|\times k$ matrix $P_\pi$ whose rows and columns are respectively indexed by vertices of $X$ and the partition $\pi$, and whose $(x,V_j)$-entry is equal to $1$ if $x\in V_j$ and $0$ otherwise. The matrix $P_\pi$ is called the \emph{characteristic matrix} of $\pi$. The next lemma shows the importance of the matrix $P_\pi$.
\begin{lem}
	If $v$ is an eigenvector of $A(X/\pi)$ corresponding to the eigenvalue $\theta$, then $P_\pi v$ is an eigenvector of $X$ with the same eigenvalue.\label{lem:eigenvector-lift}
\end{lem}
The proof of this lemma is given in \cite[Lemma~2.2.1]{godsil2016erdos}.

\subsection{Hoffman's Ratio Bound}
We recall the well-known Ratio Bound which is due to Hoffman. Let $X = (V,E)$ be a $k$-regular graph. It is well known that $k$ is an eigenvalue of $X$. The vector of all-ones, denoted by $\mathbf{1}$, is an eigenvector of $X$ corresponding to the eigenvalue $k$ and it is an easy exercise to prove that the multiplicity of $k$ is the number of components of $X$. For any $S\subset V(X)$, the \emph{characteristic vector} of $S$ is the vector $v_S \in \RR^{|V(X)|}$ which is indexed by vertices of $X$ and such that the $x$-entry of $v_S$ is equal to $1$ if $x\in S$, and $0$ otherwise. 

We state the Ratio Bound in the next theorem.
\begin{thm}[Ratio Bound]
	If $X = (V,E)$ is a $k$-regular graph with least eigenvalue $\tau$, then the independence number of $X$ satisfies
	\begin{align*}
		\alpha(X) \leq \frac{|V(X)|}{1-\frac{k}{\tau}}.
	\end{align*}
	A set $S$ is a coclique for which equality holds if and only if $v_S-\frac{|S|}{|V(X)|}\mathbf{1}$ is an eigenvector of the eigenvalue $\tau$.\label{thm:ratiobound}
\end{thm}

A proof and the history about the Ratio Bound can be found in \cite{haemers2021hoffman}.

\subsection{Graph products}

In this subsection, we will recall some important graph products that are used later in this work.
\begin{defn}
	Let $X = (V(X),E(X))$ and $Y = (V(Y),E(Y))$ be two graphs. The direct product $X \times Y$ of the graphs $X$ and $Y$ is the graph whose vertex set is $V(X) \times V(Y)$ and
	\begin{align}
		(x,y)\sim_{X\times Y} (u,v) \Leftrightarrow
			x \sim_X u \mbox{ and } y\sim_Y v.
	\end{align} 
\end{defn}
\begin{defn}
	Let $X = (V(X),E(X))$ and $Y = (V(Y),E(Y))$ be two graphs. The strong product $X \boxtimes Y$ of the graphs $X$ and $Y$ is the graph whose vertex set is $V(X) \times V(Y)$ and
	\begin{align}
		(x,y)\sim_{X\boxtimes Y} (u,v) \Leftrightarrow
		\begin{cases}
			x = u \mbox{ and } y\sim_Y v\\
			y = v \mbox{ and } x\sim_X u\\
			x \sim_X u \mbox{ and } y\sim_Y v.
		\end{cases}
	\end{align} 
\end{defn}

For any $n\geq 2$, let $I_n$ be the $n\times n$ identity matrix and $J_n$ be the $n\times n$ matrix whose entries consist of $1$. The next proposition gives the adjacency matrices of the graph products defined above.

\begin{prop}
	Let $X$ and $Y$ be two graphs with adjacency matrices $A$ and $B$, respectively. 
	\begin{enumerate}[(a)]
		\item The adjacency matrix of $X\times Y$ is $A\otimes B$. In particular, the eigenvalues of $X \times Y$ are of the form $ab$, where $a$ and $b$ are eigenvalues of $A$ and $B$, respectively.
		\item The adjacency matrix of $X\boxtimes Y$ is $(A+I_{|V(X)|})\otimes (B+I_{|V(Y)|}) - I_{|V(X)||V(Y)|}$. In particular, the eigenvalues of $X\boxtimes Y$ are of the form $(1+a)(1+b)-1$.
	\end{enumerate}
\end{prop}

Let us now introduce a new graph that is crucial to the main results of this paper.
\begin{defn}
	Let $X = (V(X),E(X))$ be a graph. The graph product $X\bowtie K_2$ is defined to be the graph obtained by taking two disjoint copies of $X$, and a vertex from one copy is adjacent to a vertex from the other copy if they are adjacent in $X \times K_2$.\label{eq:defn-graph-product}
\end{defn}

Note that the graph $X \bowtie K_2$ is exactly the graph $X \boxtimes K_2$ in which the perfect matching $\{(x,1)\sim (x,0): x \in V(X)\}$ is removed. 

The proof of the next proposition is straightforward, so we omit it.
\begin{prop}
	Let $X$ be a graph with adjacency matrix $A$. The adjacency matrix of $X\bowtie K_2$ is $J_2 \otimes A$.
\end{prop}

\section{The association scheme arising from a Gelfand pair $(\sym(n),H)$}\label{sect:association-scheme}

\subsection{Eigenvalues}

Let $G$ be a group and $H\leq G$ be a subgroup. In this subsection, we recall the formula giving the eigenvalues of each digraph in the association scheme arising from the Gelfand pair $(G,H)$. Assume that $k = \langle \ind{H}{G},\ind{H}{G}  \rangle$ and let $\mathcal{O} = \{O_0,O_1,\ldots,O_{k-1}\}$ be the set of all orbitals of $G$ acting on $G/H$. 
For any $0\leq i\leq k-1$, let $X_i$ be the orbital digraph determined by $O_i$ and  let $A(X_i)$ be its adjacency matrix. Define
\begin{align*}
	\Lambda(G,H):= \left\{ \phi \in \irr{G}:\ \left\langle \ind{H}{G},\phi \right\rangle = 1 \right\}.
\end{align*}

For any $\phi \in \Lambda(G,H)$, let $\Phi_\phi$ be the complex matrix representation that corresponds to the irreducible character $\phi$. Define
\begin{align*}
	\Psi := \sum_{\phi \in \Lambda(G,H)} \Phi_\phi.
\end{align*}
The representation $\Psi$ is the \emph{permutation representation} of $G$ acting on $G/H$. By  \cite[Theorem~3.2.2, Lemma~13.4.1]{godsil2016erdos}, the adjacency matrix $A(X_i)$ of $X_i$ commutes with $\Psi(g)$, for all $g\in G$. Note that for all $0\leq i\leq k-1$, the matrices $A(X_i)$ are diagonalizable since they can be expressed as sums of simultaneously diagonalizable matrices. Let $E_i$ be a matrix whose columns consist of pairwise linearly independent basic eigenvectors of $A(X_i)$. Then, $E_i$ is a polynomial in $A(X_i)$ and therefore, it commutes with any $\Psi(g)$, for all $g\in G$. From this, the column space of $E_i$ is an invariant subspace, i.e., a $\CC G$-module.
Further, one can also prove that every eigenspace of $A(X_i)$ is a sum of irreducible $\CC G$-modules whose characters appear in $\Lambda(G,H)$. Therefore, the eigenvalues of the graphs $(X_i)_{0\leq i\leq k-1}$ are indexed by $\Lambda(G,H)$.
For any $0\leq i\leq k-1$, let 
\begin{align}
	k_i := |\left\{ yH \in G/H  :\ (H,yH) \in O_i \right\}|.\label{eq:ki}
\end{align}
 The next theorem gives the eigenvalues of $X_i$.

\begin{thm}{\cite[Lemma~13.8.3]{godsil2016erdos}}
	The eigenvalues of $X_i$ are uniquely determined by the irreducible characters in $\Lambda(G,H)$. For any $\phi \in \Lambda(G,H)$, the eigenvalue of $X_i$ afforded by $\phi$ is given by
	\begin{align*}
		\xi_\phi(X_i) := \frac{k_i}{|H|} \sum_{h\in H} \phi(x_\ell h),
	\end{align*}
	where $x_\ell \in G$ such that $(H,x_\ell H) \in O_i$.\label{thm:general-eigenvalues}
\end{thm}

The formula for the eigenvalues given in the previous theorem is usually hard to manipulate since there is no way to control the elements of the coset $x_\ell H$, which in turn makes the corresponding character value difficult to determine. 
It is sometimes possible to compute certain eigenvalues via special techniques. In the next subsection, we will compute certain eigenvalues using equitable partitions for the case when $G = \sym(n)$.

Let $G = \sym(n)$. It is well known that the representations of $G$ are indexed by the partitions of the integer $n$. For any $\lambda \vdash n$, the corresponding irreducible $\CC G$-module is the \emph{Specht module} $S^\lambda$ and we denote its irreducible character by $\chi^\lambda$. If $(G,H)$ is a Gelfand pair, then the eigenvalues of a graph in the corresponding orbital scheme are indexed by the partitions in $\Lambda(n,H)$. For any $\lambda \in \Lambda(n,H)$, we let $\xi_\lambda$ be the eigenvalue afforded by $\chi^\lambda$. In addition, we will refer to the irreducible $\CC G$-module $S^\lambda$ as the $\lambda$\emph{-module}.

\subsection{Equitable partitions}

We assume henceforth that $G = \sym(n)$ and let $H\leq G$ such that $(G,H)$ is a Gelfand pair. Let $k$ be the rank of the group $G$ acting on $G/H$. Recall that $\mathcal{O} = \{O_0,O_1,\ldots,O_{k-1}\}$ is the set of all orbitals of $G$ acting on $G/H$ and
for any $0\leq i\leq k-1$, the digraph $X_i$ is the orbital digraph determined by $O_i$ and  the matrix $A(X_i)$ is its adjacency matrix. Recall also that $k_i$ is the degree of the orbital graph $X_i$, for $0\leq i \leq k-1$ (see \eqref{eq:ki}). Let $\lambda \vdash n$ be the partition which is the second largest in dominance ordering in $ \Lambda(n,H)$. Let $K = \sym(\lambda)$. We state the following theorem whose proof is given in  \cite[Theorem~13.9.1]{godsil2016erdos}.

\begin{lem}
	The action of $K$ on $G/H$ by left multiplication admits exactly two orbits.\label{lem:two-orbits}
\end{lem}

Using these orbits, we can sometimes easily compute certain eigenvalues of the digraphs $X_i$ defined in the previous section. Let $\Omega = G/H$. First, we note that the set of orbits $\pi = \{\mathcal{S},\Omega\setminus \mathcal{S}\}$ forms a partition of $\Omega$ which is equitable, for all orbital digraphs in the orbital scheme. The quotient matrix of each orbital digraph in the orbital scheme is given in the next lemma.
\begin{lem}
	For any $0\leq i\leq k-1$, there exists a number $0\leq a_i\leq k_i-1$ such that the quotient matrix corresponding to $X_i/\pi$ is
	\begin{align*}
		A(X_i/\pi) = 
		\begin{bmatrix}
			a_i & k_i -a_i\\
			\frac{(k_i-a_i)|\mathcal{S}|}{|\Omega|-|\mathcal{S}|} & \frac{d|\Omega|+a_i|\mathcal{S}| - 2k_i |\mathcal{S}|}{|\Omega|-|\mathcal{S}|}
		\end{bmatrix}.
	\end{align*}\label{lem:eigenvalues-equitable-partition}
\end{lem}
\begin{proof}
	Since $\pi$ is equitable with two parts, let $a_i$ be the number of vertices in $\mathcal{S}$ adjacent to a given vertex in $\mathcal{S}$. If the rows and columns of $A(X_i/\pi)$ are arranged with respect to $\{\mathcal{S},\Omega \setminus \mathcal{S}\}$, then it is clear that $A(X_i/\pi)_{11} = a_i$ and $A(X_i/\pi)_{12} = k_i-a_i$, by regularity of $X_i$. The entry $A(X_i/\pi)_{21}$ can be computed using double counting. Consider the set
	\begin{align*}
		V = \{ (xH,yH):\ xH\in \Omega\setminus \mathcal{S} \mbox{ and } yH \in \mathcal{S} \}.
	\end{align*}
	On the one hand, given $xH \in \Omega\setminus \mathcal{S}$ we have
	\begin{align*}
		A(X_i/\pi)_{21} =\left| \left\{ (xH,yH) :\ yH \in \mathcal{S} \right\} \right|.
	\end{align*}
	Moreover, 
	\begin{align*}
		|V| = \left| \bigcup_{xH\in \Omega\setminus \mathcal{S}} \{ (xH,yH):\ y H \in \mathcal{S} \} \right| = A(X_i/\pi)_{12} |\Omega\setminus \mathcal{S}| = A(X_i/\pi)_{21} \left(|\Omega|-|\mathcal{S}|\right).
	\end{align*}
	On the other hand, we have
	\begin{align*}
		|V| = \left|\bigcup_{yH\in \mathcal{S}} \left\{(xH,yH):\ xH\in \Omega\setminus \mathcal{S} \right\}\right| = |\mathcal{S}|A(X_i/\pi)_{12} = |\mathcal{S}|(k_i-a_i).
	\end{align*}
	Therefore, we have $A(X_i/\pi)_{21} = \frac{|\mathcal{S}|(k_i-a_i)}{|\Omega| - |\mathcal{S}|}$. The entry $A(X_i/\pi)_{22} $ can be easily deduced since it is equal to $k_i - A(X_i/\pi)_{21}$.
\end{proof}

By Theorem~\ref{thm:eigenvalues-equitable-partitions}, the eigenvalues of $A(X_i/\pi)$ are also eigenvalues of the digraph $X_i$, for any $0\leq i\leq k-1$. The eigenvalues of $A(X_i/\pi)$ are   
\begin{align*}
	k_i\mbox{ and } -\frac{(k_i-a_i)|\mathcal{S}|}{|\Omega|-|\mathcal{S}|}.
\end{align*}
We now present some consequences of Lemma~\ref{lem:eigenvalues-equitable-partition} that are crucial to the main result of this paper.
\begin{lem}
	Let $\Omega = G/H$ and $\pi = \left\{ \mathcal{S},\Omega \setminus \mathcal{S} \right\}$ be the set of orbits of $K=\sym(\lambda)$. For $0\leq i \leq k-1$, the set $\mathcal{S}$ is a coclique of $X_i$ if and only if $a_i = 0$.
\end{lem}

The proof of this lemma is straightforward, thus it is omitted. 

The next theorem gives more details on the relation between cocliques from orbit partitions and the $\lambda$-module, whenever $\lambda = [n-1,1]$ is the second largest in dominance ordering in $\Lambda(n,H)$.

\begin{thm}
	Let $G = \sym(n)$ and $(G,H)$ be a Gelfand pair such that $[n-1,1]$ is the second largest in dominance ordering in $\Lambda(n,H)$. Let $\Omega = G/H$ and $\pi = \left\{ \mathcal{S},\Omega \setminus \mathcal{S} \right\}$ be the set of orbits of $K = \sym([n-1,1])$. For any $0\leq i \leq k-1$, if the set $\mathcal{S}$ is a coclique of $X_i$, then $-\frac{k_i|\mathcal{S}|}{|\Omega| - |\mathcal{S}|}$ is the eigenvalue afforded by the $[n-1,1]$-module.\label{thm:cocliques-module}
\end{thm}
\begin{proof}

	We will use the same idea as the proof in \cite{godsil2016algebraic}. We assume that $\mathcal{S}$ is a coclique of $X_i$, for some $0\leq i\leq k-1$. Let $P = P_\pi$ be the characteristic matrix of the equitable partition $\pi$. That is, $P$ is the $|\Omega| \times 2$ matrix whose rows are indexed by $\Omega$ and whose columns are indexed by $\pi$ such that $P(x,X) = 1$ if $x\in X$ and $0$ otherwise, for $x\in \Omega$ and $X \in \{\mathcal{S},\Omega\setminus \mathcal{S}\}$. By Lemma~\ref{lem:eigenvector-lift}, if $v$ is an eigenvector of $A(X_i/\pi)$ with eigenvalue $-\frac{k_i|\mathcal{S}|}{|\Omega|-|\mathcal{S}|}$, then $Pv$ is an eigenvector of $A(X_i)$ with the same eigenvalue. A basic eigenvector of $A(X_i/\pi)$ corresponding to $-\frac{k_i|\mathcal{S}|}{|\Omega|-|\mathcal{S}|}$ is
	\begin{align*}
		v = 
		\begin{bmatrix}
			1-\frac{|\mathcal{S}|}{|\Omega|}\\
			-\frac{|\mathcal{S}|}{|\Omega|}
		\end{bmatrix}.
	\end{align*}
	For any $xH \in \Omega$, it is not hard to verify that 
	\begin{align*}
		Pv(xH)
		=
		\begin{cases}
			1-\frac{|\mathcal{S}|}{|\Omega|} &\mbox{ if } xH\in \mathcal{S}\\
			- \frac{|\mathcal{S}|}{|\Omega|} &\mbox{ otherwise}
		\end{cases}.
	\end{align*}
	In other words, $w:=Pv = v_\mathcal{S} - \frac{|\mathcal{S}|}{|\Omega|}\mathbf{1}$, where $v_\mathcal{S}$ is the characteristic vector of $\mathcal{S}$. For any $\sigma \in G = \sym(n)$, define $\sigma(w):= v_{\sigma(\mathcal{S})}-\frac{|\mathcal{S}|}{|\Omega|} \mathbf{1}$. This action is well defined and it makes the subspace 
	\begin{align*}
		V := \left\{ \sigma(w) :\ \sigma \in \sym(n) \right\}
	\end{align*}
	into a $\CC G$-module. Moreover, the set $V$ is a subspace of the eigenspace of $X_i$ corresponding to the eigenvalue $-\frac{k_i|\mathcal{S}|}{|\Omega|-|\mathcal{S}|}$. We claim that $V$ is in fact the $[n-1,1]$-module. 
	
	Let $W$ be the $\CC G$-module corresponding to the character $\ind{K}{G} = \chi^{[n]} + \chi^{[n-1,1]}$. It is clear that $W = S^{[n]} \oplus S^{[n-1,1]}$. The $\CC G$-module $W$ can be identified with the set of all functions $L(\sym(n))$ that are constant on $K$. A basis of the vector space $L(\sym(n))$ can be easily determined. For any $xK \in G/K$, define
	\begin{align*}
		\delta_{xK}(\sigma) :=
		\begin{cases}
			1 &\mbox{ if }\sigma \in xK\\
			0 & \mbox{ otherwise}.
		\end{cases}
	\end{align*}
	
	One can verify that $\{ \delta_{xK} :\ xK\in G/K \}$ is a basis of $W$. Now, define the map $f:V \to W$ by
	\begin{align*}
		 f\left(\sigma(w)\right) := \delta_{\sigma K} - \frac{|\mathcal{S}|}{|\Omega|} \sum_{xK \in G/K} \delta_{\sigma xK}.
	\end{align*}
	The map $f$ is well defined since 
	\begin{align*}
		\mbox{$\sigma (w) = \tau(w) \Leftrightarrow$  $v_{\sigma(\mathcal{S})} = v_{\tau(S)} \Leftrightarrow \tau^{-1}\sigma (\mathcal{S}) = \mathcal{S}$.}
	\end{align*}
	Hence, $\tau^{-1}\sigma$ leaves the partition $\pi$, which consists of orbits of $K$, invariant. Therefore, $\sigma K = \tau K$, which translates into $f$ being well defined. Moreover, the map $f$ is a $\CC G$-module homomorphism. Hence, $V$ is isomorphic to a submodule of $W = S^{[n]}\oplus S^{[n-1,1]}$. As $V\neq 0$ and $\dim (V) >1$, it is clear that $V$ is isomorphic to $S^{[n-1,1]}$.
\end{proof}

\section{Proof of Theorem~\ref{thm:main1}}\label{sect:proof1}
Let $G = \sym(2k)$ and $H = \sym(k) \wr \sym(2)$. The rank of the group $G$ acting on $G/H$ is $\lfloor\frac{k}{2}\rfloor+1$ since by \cite{godsil2010multiplicity}, we have
\begin{align*}
	\ind{H}{G} = \sum_{i = 0}^{\lfloor \frac{k}{2} \rfloor} \chi^{[2k-2i,2i]}.
\end{align*}
Therefore, the second largest in dominance ordering in $\Lambda(n,H)$ is $[2k-2,2]$. We note that the combinatorial objects that correspond to the cosets of $H$ in $G$ are the \emph{uniform partitions} of $[2k]$ into $2$ blocks of size $k$. That is, the collections
\begin{align*}
	U_{k} := \left\{ \{B_1,B_2\} :\ B_1\cup B_2 = [2k], \ B_1\cap B_2 = \varnothing,\  |B_1| = |B_2| = k \right\}.
\end{align*}
The action of the symmetric group $\sym(2k)$ on $[2k]$ induces an action on $U_k$. It is not hard to see that $\sym(k) \wr \sym(2)$ is the stabilizer of the partition $\{\{1,2,\ldots,k\},\{k+1,k+2,\ldots,2k\}\} \in U_k$. 

For any $A \in \binom{[2k]}{k}$, we let $\underline{A}$ be the complement of $A$ in $[2k]$.

\begin{lem}
	The orbitals of $G = \sym(2k)$ in its action on $U_k$ are of the form
		\begin{align*}
			\mathcal{O}_i = \left\{ \left(\mathcal{B},\mathcal{B}^\prime\right) \in U_{k} \times U_k :\ \mbox{there exists $B\in \mathcal{B},\ B^\prime \in \mathcal{B}^\prime$ such that $|B\cap B^\prime| = k-i$} \right\},
		\end{align*}
		for $0\leq i \leq \lfloor \frac{k}{2} \rfloor$.
\end{lem}
\begin{proof}
	It is worth noting that if two partitions $\mathcal{B}= \{B,\underline{B}\}$ and $\mathcal{B}^\prime = \{ B^\prime,\underline{B}^\prime \}$ of $U_{k}$ are such that $|B\cap B^\prime| = k-i$, then we must have $|B\cap \underline{B}^\prime | = i$. 
	
	For any $0\leq i \leq k$, let $X_i = \{ 1,2,\ldots,k-i,k+1,k+2,\ldots,k+i \}$. Note that $X_0 = \{ 1,2,\ldots,k \}$ and $X_k = \{ k+1,k+2,\ldots,2k \}$. Let $\ell = \lfloor\frac{k}{2}\rfloor$ and for $0 \leq i \leq \ell$, define $\mathcal{B}_i := \{ X_i,\underline{X}_i \}$. We claim that $\{ \mathcal{B}_i: 0\leq i\leq \ell \}$ consists of a complete list of representatives from distinct suborbits of $\stab{\sym(2k)}{\mathcal{B}_{0}} = \sym(k) \wr \sym(2)$. It is clear that $X_0$ is in a single orbit of $\sym(k)\wr \sym(2)$. Assume that there exists $i< j$ in the set $\{1,\ldots,\ell\}$ such that $\mathcal{B}_i$ and $\mathcal{B}_j$ are in the same orbit of $\sym(k) \wr \sym(2)$. If $\sigma\in \sym(k) \wr \sym(2)$ such that $\sigma(\mathcal{B}_i) = \mathcal{B}_j$, then either $\sigma(X_i) = X_j$ or $\sigma(X_i) = \underline{X_j}$. 
	\\ \\
	{\it Case~1:} Assume that $\sigma(X_i) = X_j$. If $\sigma(X_0) = X_0$, then $\sigma$ must map the elements $X_i\cap X_0 = \{1,2,\ldots,k-i\}$ to $X_j\cap X_0 =\{ 1,2,\ldots,k-j \}$. By the Pigeonhole Principle, this cannot happen since $i<j$.
	Similarly, if $\sigma(X_0) = \underline{X_0} = X_k$, then $\sigma$ has to map $\{k+1,k+2,\ldots,k+i\}$ to $\{ 1,2,\ldots,k-j \}$. This is only possible if $k-j = i$. As $i$ and $j$ are distinct and at most $\ell \leq\frac{k}{2}$, we conclude this case also cannot happen.
	\\ \\
	{\it Case~2:} Assume that $\sigma(X_i) = \underline{X_j}$. Note that $\underline{X_j} = \{ k-j+1,\ldots,k,k+j+1,\ldots,2k \}$. By a careful analysis of the image of $X_0$ by $\sigma$, we also conclude that this case is not possible. 
	
	Consequently, the set $\{ \mathcal{B}_i : 0\leq i\leq \ell \}$ consists of elements in different orbits of $\sym(k) \wr \sym(2)$. Since the cardinality of this set and the rank of $G$ acting on $U_k$ coincide, we conclude that the orbitals of $G$ on $U_k$ are those listed in the statement of the lemma. We obtain the rest by making $\sym(2k)$ act on every representative of the orbitals.
\end{proof}

Consequently, all orbital digraphs are symmetric, implying that they are all undirected graphs. Using this lemma, we can show that the orbits of $K = \sym([2k-2,2])$ are not cocliques of any orbital graph in this association scheme. The orbits of $K$ are $\mathcal{S}$ and $U_{k} \setminus \mathcal{S} $, where 
\begin{align*}
	\mathcal{S} = \left\{ \{B,\underline{B}\} :\ 2k-1,2k \in B \right\}.
\end{align*}
In the next theorem, we claim that neither of $\mathcal{S}$ nor $U_k\setminus\mathcal{S}$ is a coclique of a graph in the orbital scheme corresponding to $\sym(2k)$ acting on $U_k$.
\begin{thm}
	The collection $\mathcal{S}$ is not a coclique of the orbital graph corresponding to $\mathcal{O}_i$, for any $0\leq i \leq \lfloor\frac{k}{2}\rfloor$.
\end{thm}
\begin{proof}
	Let $1\leq i\leq \ell = \lfloor\frac{k}{2} \rfloor$. In the orbital graph of $\mathcal{O}_i$, by definition, there is an edge between $\{ B_1,B_2 \}$ and $\{B_1^\prime,B_2^\prime\}$ if $|B_1 \cap B_1^\prime| \in \{i,k-i\}$. Consider the sets 
	\begin{align*}
		A &= \{ 1,2,\ldots,k-i-2,k+1,k+2,\ldots,k+i,2k-1,2k \} \\
		B &= \{1,2,\ldots,k-i -2,k+i+1,k+i+2,\ldots,k+2i,2k-1,2k \}.
	\end{align*}
	Since $|A\cap B| = k-i$, the partitions $\{A,\underline{A}\}$ and $\{B,\underline{B}\}$ are adjacent in the orbital graph of $\mathcal{O}_i$. As these partitions are in $\mathcal{S}$, we conclude that $\mathcal{S}$ is not a coclique of the orbital graph of $\mathcal{O}_i$, for any $1\leq i\leq \ell$. {This completes the proof.}
\end{proof}
\begin{rmk}
	The orbital graph of $\mathcal{O}_1$ is called the folded Johnson graph $\overline{J}(2k,k)$ (see \cite[Section~9.1]{brouwer1989distance} for details). This graph encodes certain combinatorial objects as its cocliques. Two partitions $\mathcal{B}$ and $\mathcal{B}^\prime$ of $U_k$ are called partially $2$-intersecting if there exists $B \in \mathcal{B}$ and $B^\prime \in \mathcal{B}^\prime$ such that $|B\cap B^\prime| \geq 2$. It is straightforward to verify that a maximum partially $2$-intersecting family of $U_k$ is a coclique in the orbital graph of $\mathcal{O}_1$ (i.e., the folded Johnson graph $\overline{J}(2k,k)$), and vice versa. As far as we know, there is nothing known about the cocliques of this graph.  
\end{rmk}

\section{The orbital scheme from line~1}\label{sect:line1}
In this section, we will consider the Schurian association scheme obtained from the action of $\sym(n)$ on the cosets of $\sym(k) \times \sym(n-k)$. We will exhibit an action on certain combinatorial objects that is permutation equivalent to the latter. To this end, we will recall a way of constructing the $k$-subsets of $[n]$ with an equivalence relation.

Let $n$ be a positive integer and $2\leq k\leq n$. Consider the $\mathcal{I}_{n,k}:= \left\{ f:[k] \to [n] :\ f\mbox{ is injective} \right\}$. An element of $\mathcal{I}_{n,k}$ can be represented as a sequence of the form $(a_1,a_2,\ldots,a_k)$. Define the relation $\mathcal{R}$ on $\mathcal{I}_{n,k}$ such that $(a_1,a_2,\ldots,a_k) \mathcal{R} (b_1,b_2,\ldots,b_k)$ if there exists $\sigma \in \sym(k)$ such that
\begin{align*}
	 a_i = b_{\sigma(i)}, \mbox{ for } i\in \{1,2,\ldots,k\} .
\end{align*}
It is not hard to prove that the relation $\mathcal{R}$ is in fact an equivalence relation. The equivalence classes of $\mathcal{R}$ are of the form
\begin{align*}
	\overline{(a_1,a_2,\ldots,a_k)} = \left\{ (a_{\sigma(1)},a_{\sigma(2)},\ldots,a_{\sigma(k)}) \in \mathcal{I}_{n,k} :\ \sigma \in \sym(k) \right\}.
\end{align*}
These equivalence classes are exactly the $k$-subsets of $[n]$. We will denote the equivalence class $\overline{(a_1,a_2,\ldots,a_k)}$ by $\{a_1,a_2,\ldots,a_k\}$ from now on. We will also define $\binom{[n]}{k}$ to be the set of all $k$-subsets of $[n]$.

Given a $k$-subset $A = \{a_1,a_2,\ldots,a_k\}$ of $[n]$, it is clear that the setwise stabilizer of $A$ in the symmetric group $\sym(n)$ is equal to $\sym(A) \times \sym([n]\setminus A ) \cong \sym(k) \times \sym(n-k)$. In fact, it is not hard to see that the action of $\sym(n)$ on the cosets of $\sym(k) \times \sym(n-k)$ is permutation equivalent to the action of $\sym(n)$ on the $k$-subsets of $[n]$.

The association scheme which arises from the action of $\sym(n)$ on the $k$-subsets of $[n]$ is well known and well studied. The orbitals of this action are
\begin{align}
	\mathcal{O}_i = \left\{ (A,B) :\ \mbox{$A,B \in \binom{[n]}{k}$ such that } |A\cap B| = k-i \right\},
\end{align}
for $0\leq i\leq k$. The corresponding association scheme is the \emph{Johnson scheme} $\mathcal{J}(n,k)$. The permutation character corresponding to the action of $\sym(n)$ on the $k$-subsets is given by
\begin{align}
	\ind{\sym(k) \times \sym(n-k)}{\sym(n)} = \sum_{i=0}^k \chi^{[n-i,i]}.
\end{align}
Hence, the second largest partition of $\Lambda(n,\sym(k)\times \sym(n-k))$ is equal to $[n-1,1]$. The orbits of the Young subgroup $\sym([n-1,1])$ on the $k$-subsets are $S$ and $\binom{[n]}{k}\setminus S$, where
\begin{align}
	S&= \left\{ A \in \binom{[n]}{k} :\ n\in A \right\}.\label{eq:max-coclique1}
\end{align}

The graph obtained from the orbital $\mathcal{O}_k$ has the property that two $k$-subsets of $[n]$, $A$ and $B$, are adjacent if $A\cap B = \varnothing$. This graph is known as the \emph{Kneser graph} $K(n,k)$. By the well-known Erd\H{o}s-Ko-Rado theorem, the set $S$ in \eqref{eq:max-coclique1} is a maximum coclique of $K(n,k)$. It is also well known that the eigenvalue afforded by the $[n-1,1]$-module is the least eigenvalue of $K(n,k)$, which is equal to 
\begin{align*}
	-\frac{\binom{n-k}{k}\binom{n-1}{k-1}}{\binom{n-1}{k}}.
\end{align*}
\section{The orbital scheme from line~2 and line~3}\label{sect:line2-3}
In this section, we prove that the answer to Question~\ref{qst:main} is affirmative for the Gelfand pair $(G,H)$, where $G = \sym(n)$ and $H \in \left\{ \alt(k) \times \sym(n-k),\sym(k) \times \alt(n-k) \right\}$. We will only give the proof for the case $H = \alt(k) \times \sym(n-k)$ since the other case is similar.

\subsection{The quasi $k$-subsets of $[n]$}
Let $n$ be a positive integer and $k\leq n$. Recall that $\mathcal{I}_{n,k}$ is the set of all injective maps from $[k]$ to $[n]$. Define the relation on $\mathcal{I}_{n,k}$ by 
\begin{align}
	(a_1,a_2,\ldots,a_k) \mathcal{R} (b_1,a_2,\ldots,b_k) \Leftrightarrow \exists \sigma \in \alt(k) \mbox{ such that } b_i = a_{\sigma(i)},
\end{align}
for all $i\in \{1,2,\ldots,k\}$. It is clear that $\mathcal{R}$ is an equivalence relation on $\mathcal{I}_{n,k}$. For any $(a_1,a_2,\ldots,a_k)\in \mathcal{I}_{n,k}$, define the corresponding equivalence class of $\mathcal{R}$ to be
\begin{align*}
	\overline{\left(a_1,a_2,\ldots,a_k\right)} &:= \left\{ \left(a_{\sigma(1)},a_{\sigma(2)},\ldots,a_{\sigma(k)}\right) :\ \sigma \in \alt(k) \right\}.
\end{align*}

Next, let us compute the number of equivalence classes of $\mathcal{R}$. We claim that each $k$-subset $\{a_1,a_2,\ldots,a_k\}$ determines two equivalence classes of $\mathcal{R}$. To see this, let $T$ be the set of all injective maps from $[k]$ to $A = \{a_1,a_2,\ldots,a_k\}$ (i.e., bijections). The group $\sym(k)$ acts naturally on $T$ by permuting the entries. By the orbit-counting lemma on the action of $\alt(k)$ on $T$, we  know that
\begin{align*}
	\frac{2}{k!} \sum_{\sigma \in \alt(k)} \left|\left\{ (b_1,b_2,\ldots,b_k) \in T :\ b_{\sigma(i)} = b_i, \forall i\in \{ 1,2,\ldots,k \} \right\}\right| = \frac{2}{k!} k! = 2.
\end{align*}
In other words, the alternating group $\alt(k)$ admits two orbits on $T$. The two equivalence classes of $\mathcal{R}$ determined by $A = \{a_1,a_2,\ldots,a_k\}$ are
\begin{align*}
	\{a_1,a_2,a_3\ldots,a_k\}^+:=\overline{(a_1,a_2,a_3\ldots,a_k)} \mbox{ and } \{a_1,a_2,a_3\ldots,a_k\}^- := \overline{(a_2,a_1,a_3,\ldots,a_k)}.
\end{align*}
 We will call the collection of all equivalence classes of $\mathcal{R}$ the \emph{quasi $k$-subset} of $[n]$ and we will denote it by $\binom{[n]}{k}^{\pm}$. The classes $\{a_1,a_2,a_3\ldots,a_k\}^+$ and $\{a_1,a_2,a_3\ldots,a_k\}^-$ are called the even and odd quasi $k$-subsets of $\{a_1,a_2,a_3\ldots,a_k\}$, respectively. Note that the number of quasi $k$-subsets of $[n]$ is $2\binom{n}{k}$.

\subsection{Action of the $\sym(n)$ on quasi $k$-subsets}

Let $n\geq 3$. Given a $k$-tuple $A = (a_1,a_2,\ldots,a_k)\in \mathcal{I}_{n,k}$ and $\sigma \in \sym(n)$, the permutation $\sigma$ acts naturally on $A$ via $\sigma (A):= \left(\sigma(a_1),\sigma(a_2),\ldots,\sigma(a_k)\right)$. Assume that the elements of $A$ are ordered in an increasing way, i.e., $a_1<a_2<\ldots<a_k$.
\begin{itemize}
	\item For any $i\in \{1,2,\ldots,k\}$, define 
	\begin{align}
		\mathsf{M} (\sigma(A),i):= \min_{i\leq j\leq k} \left\{ \sigma(a_j) :\ \sigma(a_j) < \sigma(a_i) \right\}.
	\end{align}
	\item Consider now the sequences
	\begin{align*}
		\begin{cases}
			\tau_0(A) &= \sigma(A)\\
			\tau_i(A) &= \mu(A,i) \left(\tau_{i-1}(A)\right), \mbox{ for } 1\leq i\leq k,
		\end{cases}
	\end{align*}
	where $\mu(A,i):=\left(\mathsf{M}(\tau_{i-1}(A),i)\ \tau_{i-1}(a_i)\right)$ is the transposition that interchanges $\mathsf{M}(\tau_{i-1}(A),i) $ and $ \tau_{i-1}(a_i)$. 
	Now, define the product of transpositions
	\begin{align}
		\mathsf{sort}(\sigma(A)) := \mu(A,k) \ldots \mu(A,2)\mu(A,1). 
	\end{align}
	It is not hard to see that the permutation $\mathsf{sort}(\sigma(A))$ sorts the entries of $\sigma(A)$ into an increasing order. 
	\item
	Recall that the homomorphism $\mathsf{sgn}$ is the unique homomorphism from $\sym(n)$ to $\{1,-1\}$ which assigns even permutations to $1$ and odd permutations to $-1$. For any $\sigma\in\sym(n)$ and $A = (a_1,a_2,\ldots,a_k)$, define
	\begin{align}
		\sgn{\sigma(A)} &:= 
		\begin{cases}
			+ & \mbox{ if }\sgn{\mathsf{sort}(\sigma(A))} = 1,\\
			- & \mbox{ otherwise.}
		\end{cases}
	\end{align}
	\item Given a set $S = \{a_1,a_2,\ldots,a_k \}$, define $\hat{S} := (a_{i_1},a_{i_2},\ldots,a_{i_k}) \in \mathcal{I}_{n,k}$ such that $\{i_1,i_2,\ldots,i_k\} = \{1,2,\ldots,k\}$ and $a_{i_j}< a_{i_{j+1}}$, for any $j\in \{1,2,\ldots,k-1\}$.
\end{itemize}

We can make $\sym(n)$ act on quasi $k$-subsets of $[n]$ as follows. For any $\sigma\in \sym(n)$ and a $k$-subset $S\in \binom{[n]}{k}$, we let
\begin{align}
	\sigma(S^*) := 
	\begin{cases}
		\sigma(S)^{\sgn{\sigma(  \hat{S}     )}} & \mbox{ if }*=+,\\
		\sigma(S)^{-\sgn{\sigma(\hat{S})}} & \mbox{ if }*=-, 
	\end{cases}\label{eq:action}
\end{align}
with the convention that $(-)(-) =+, \ (+)(+) = +,\ (-)(+) = (+)(-) = -$. We let the reader verify that this induces a group action of $\sym(n)$ on quasi $k$-subsets of $[n]$.

Let $A = \{1,2,\ldots,k\}$. Next, we determine the stabilizer of $A^+$ in $\sym(n)$. If $\sigma\in\sym(n)$ fixes $A^+$, then it must fix the set $A = \{1,2,\ldots,k\}$ setwise. Hence, $\sigma \in \sym(k) \times \sym(n-k)$. By definition, we have $\sigma(A^+) = (\sigma(A))^{\sgn{\sigma(\hat{A})}}$. On the one hand, we have
\begin{align*}
		\sigma(\hat{A})=	\sigma((1,2,\ldots,k)) = \sigma_{|\{1,2\ldots,k\}}((1,2,\ldots,k)) =  (\sigma(1),\sigma(2),\ldots,\sigma(k)).
\end{align*}
On the other hand, if $\tau = \mathsf{sort}(\sigma(\hat{A}))$, then we have 
\begin{align*}
	 \tau((\sigma(1),\sigma(2),\ldots,\sigma(k))) =  (1,2,\ldots,k) = \hat{A}.
\end{align*}
Since $\sigma_{|\{1,2,\ldots,k\}}$ and $\tau$ are permutations of $\{1,2,\ldots,k\}$, we have that $\sigma_{|\{1,2,\ldots,k\}} = \tau^{-1}$.
If $\sigma_{|\{1,2,\ldots,k\}}$ is an odd permutation, then it is clear that $\sigma(A^+) = A^-$ and $\sigma(A^-) = A^+$. If $\sigma_{|\{1,2,\ldots,k\}}$ is an even permutation, then $\sigma(A^+) = A^+$ and $\sigma(A^-) = A^-$. Therefore, 
\begin{align*}
	\stab{\sym(n)}{A^+} = \stab{\sym(n)}{A^-}=  \alt(k) \times \sym(n-k).
\end{align*}
We state the following result whose proof is omitted.
\begin{lem}
	The action of $\sym(n)$ on the cosets of $\alt(k)\times \sym(n-k)$ is permutation equivalent to the action of $\sym(n)$ on quasi $k$-subsets of $[n]$.
\end{lem}

\subsection{The quasi Johnson scheme}
In this subsection, we describe the association scheme obtained from the action of $\alt(k) \times \sym(n-k)$. We first characterize all orbitals of the action of $\sym(n)$ on quasi $k$-subsets of $[n]$. 

\begin{thm}
	An orbital of $\sym(n)$ acting on the quasi $k$-subsets is one of the following
	\begin{enumerate}[1)]
		\item $\mathcal{O}_0^+ := \left\{(A^+,A^+) :\ A\in \binom{[n]}{k} \right\}$,
		\item $\mathcal{O}_0^-:=\left\{(A^+,A^-):\ A\in \binom{[n]}{k}\right\}$,
		\item $\mathcal{O}_1^-:=\left\{ (A^+,B^-) \in \binom{[n]}{k}^\pm \times \binom{[n]}{k}^\pm :\ |A\cap B| = k-1 \right\}$,
		\item $\mathcal{O}_1^+:=\left\{ (A^+,B^+) \in \binom{[n]}{k}^\pm \times \binom{[n]}{k}^\pm :\ |A\cap B| = k-1 \right\}$, 
		\item or one of 
		\begin{align*}
			\mathcal{O}_i = \left\{ (A^+,B^+) \in \binom{[n]}{k}^\pm \times \binom{[n]}{k}^\pm :\ |A\cap B| = k-i \right\},
		\end{align*}
	for	$2\leq i\leq k$.
	\end{enumerate}\label{thm:orbitals}
\end{thm}
\begin{proof}
	By \cite{godsil2010multiplicity}, the permutation character of the group $\sym(n)$ acting on quasi $k$-subsets of $[n]$ is
	\begin{align*}
		\ind{\alt(k) \times \sym(n-k)}{\sym(n)} &= \sum_{i=0}^{k} \chi^{[n-i,i]} + \chi^{[n-k,1^{k}]} + \chi^{[n-k+1,1^{k-1}]}.
	\end{align*}
	Hence, the number of orbitals of the corresponding group action is $k+3.$ 
	
	Let $X = \{1,2,3,\ldots,k\}$. Consider some distinct elements $x_1<x_2<\ldots<x_{k}$ of the set $[n] \setminus \{1,2,\ldots,k\}$. For any $1\leq i\leq k$, we define $$X_i = \{1,2,\ldots,k-i,x_1,x_2,\ldots,x_{i}  \}.$$ 
	
	It is not hard to verify that $(X^+,X^+),\ (X^+,X_1^+),\ (X^+,X_1^-),$ $(X^+,X^-)$, and $(X^+,X_i^+)$, for $2\leq i\leq k$, are elements of the five types of orbitals listed in the statement of the theorem.
	To prove the conclusion of the theorem, it is enough to prove that $S = \{X^+,X_1^+,X_1^-,X^-\} \cup \{X_2^+,X_3^+,\ldots,X_k^+\}$ are in different suborbits of $\stab{\sym(n)}{X^+} = \alt(k) \times \sym(n-k)$. 
	
	It is clear that $X^+$ cannot be in the same orbit of $\stab{\sym(n)}{X^+}$ as any other element of $S$. We claim that for any $A^*,B^\dagger\in S$, where $*,\dagger \in \{+,-\}$, such that $|A\cap X| \neq |B\cap X|$, there exist no $\sigma \in \alt(k) \times \sym(n-k)$ such that $\sigma(A^*) = B^\dagger$. This is clear since assuming that $|A\cap X| < |B\cap X|$, there would be an element of $X\setminus A$ that is mapped to an element of $X$ by such permutation, which is not possible for any element $\sigma$ of $\alt(k) \times \sym(n-k)$.
	
	By this claim, it is enough to verify the cases where the intersections with $X$ have the same size, i.e., those with the same index.
	
	We claim that $X_1^+$ and $X_1^-$ are not in the same suborbit. To see this, assume that there exists $\sigma \in \alt(k) \times \sym(n-k)$ such that $\sigma(X_1^+) = X_1^-$. As $\sigma(X^+) = X^+$, we must have that $\sigma_{|X}$ is an even permutation and also that $\sigma(x_1) = x_1$ and $\sigma(k) = k$. Hence, $\sigma_{|\{1,2,\ldots,k-1\}} = \sigma_{|X_1}$ is an even permutation, which makes $\sigma(X_1^+) = X_1^-$ impossible. 
	
	Consequently, the set $\{ (X^+,A) :\ A\in S \}$ is a complete set of representatives of orbitals of $\sym(n)$ acting on the quasi $k$-subsets of $[n]$. It is straightforward to verify that the elements of this set yields the five lists of orbitals. This completes the proof.
\end{proof}

It is not hard to see that the orbitals are all self-paired. From Theorem~\ref{thm:orbitals}, we can explicitly define the Schurian association scheme corresponding to the action of $\sym(n)$ on the cosets of the multiplicity-free subgroup $\alt(k) \times \sym(n-k)$. The association scheme corresponding to this action is the \emph{quasi Johnson scheme}, denoted by $\mathcal{J}^{+\pm}(n,k)$, which is determined by the relations
\begin{align*}
	\left\{ \mathcal{O}_0^\pm,\mathcal{O}_1^\pm,\mathcal{O}_2,\mathcal{O}_3,\ldots,\mathcal{O}_k \right\}.
\end{align*}

\subsection{The graphs in $\mathcal{J}^{+\pm}(n,k)$}
For $2\leq i\leq k$, we define the graph $J^\pm(n,k,i)$ to be the graph on $\binom{[n]}{k}^\pm$ whose edge set is $\mathcal{O}_i$. We will call the graph $\mathcal{K}^\pm(n,k) := J^\pm(n,k,k)$ the \emph{quasi Kneser graph}. The graph determined by $\mathcal{O}_0^-$ is a perfect matching. We will denote the latter by $J^-(n,k,0)$. The graphs determined by $\mathcal{O}_1^+$ and $\mathcal{O}_1^-$ are called the \emph{quasi Johnson graph of type $+$} and \emph{quasi Johnson graph of type $-$}, respectively and we will denote them by $J^+(n,k,1)$ and $J^-(n,k,1)$, respectively.

In this subsection, we determine the structure of some  graphs in the quasi Johnson scheme $\mathcal{J}^{+\pm}(n,k)$ in terms of a graph product.
We prove the next proposition.
\begin{prop}
	For any $2\leq i\leq k$, we have $J^\pm(n,k,i) = K_2 \bowtie J(n,k,i)$.
\end{prop}
{
\begin{proof}
	Let $i\in \{2,3,\ldots,k\}$, $V^+ = \{ A^+ :\ A\in \binom{[n]}{k}  \}$ and $V^- = \{ A^- :\ A\in \binom{[n]}{k} \}$. Since $(A^+,B^+),(A^-,B^-) \in \mathcal{O}_i$ for any $A,B\in \binom{[n]}{k}$ such that $|A\cap B| = k-i$, it is clear that the subgraph induced by $V^+$ and $V^-$ are both isomorphic to $J(n,k,i)$. The edges of the form $(A^+,B^-)$ are exactly those in $K_2 \times J(n,k,i)$. 
	Hence, $J^\pm(n,k,i) = K_2 \bowtie J(n,k,i)$.
\end{proof}}

\begin{cor}
	Let $2\leq i\leq k$. The eigenvalues of $J^\pm(n,k,i)$ are equal to $0$ or of the form $2\lambda$, where $\lambda$ is an eigenvalue of $J(n,k,i)$.
\end{cor}
\begin{proof}
	Let $A_{n,k,i}$ and $A_{n,k,i}^\pm$ be the adjacency matrices of the graphs $J(n,k,i)$ and $J^\pm(n,k,i)$, respectively. Since $J^\pm (n,k,i) = K_2 \bowtie J(n,k,i)$, the adjacency matrix of $J^\pm (n,k,i)$ is 
	\begin{align*}
		A^\pm_{n,k,i} =
		\begin{bmatrix}
			A_{n,k,i} & A_{n,k,i}\\
			A_{n,k,i} & A_{n,k,i}
		\end{bmatrix}
		= 
		\begin{bmatrix}
			1 & 1 \\
			1 & 1
		\end{bmatrix} 
		\otimes 
		A_{n,k,i}.
	\end{align*}
	Since the spectrum of $2\times 2$ all-ones matrix is $\{0,2\}$, the result follows.
\end{proof}
\begin{cor}
	The eigenvalues of the quasi Kneser graph $\mathcal{K}^\pm(n,k)$ are $0$ or
	\begin{align*}
		2(-1)^j\binom{n-k-j}{k-j},
	\end{align*}
	for $0\leq j\leq k$.\label{cor:eigenvalues-quasi-kneser}
\end{cor}

We could not determine the eigenvalues of the quasi Johnson graph $J^\pm(n,k,1)$, and we do not know whether these graphs have appeared in the literature already. 

\subsection{The $[n-1,1]$-module}

Let $K = \sym([n-1,1])$ and let $\Omega = \binom{[n]}{k}^\pm$. The group $K$ has two orbits on the quasi $k$-subsets of $[n]$. These two orbits of $K$ are
\begin{align*}
	\mathcal{S} =  \left\{ A^\pm :\ A\in \binom{[n]}{k} \mbox{ and } n\in A \right\}\mbox{ and } \Omega \setminus \mathcal{S}.
\end{align*}

\begin{prop}
	The family $\mathcal{S}$ is a maximum coclique of $\mathcal{K}^\pm (n,k)$.
\end{prop}
\begin{proof}
	Since $|A\cap B| = 1$, we conclude that $(A^\pm,B^\pm) \not \in \mathcal{O}_k$, for any $A,B\in \binom{[n]}{k}$. Therefore, $\mathcal{S}$ is a coclique of $\mathcal{K}^\pm(n,k)$.
	
	To prove that $\mathcal{S}$ is a maximum coclique, we use the Ratio Bound. By Corollary~\ref{cor:eigenvalues-quasi-kneser}, the largest and the least eigenvalues of $\mathcal{K}^\pm(n,k)$ are respectively
	\begin{align*}
		2\binom{n-k}{k} \mbox{ and } -2\binom{n-k-1}{k-1}.
	\end{align*}
	By the Ratio Bound, we have
	\begin{align*}
		\alpha(\mathcal{K}^\pm (n,k)) \leq \frac{2\binom{n}{k}}{1-\frac{2\binom{n-k}{k}}{-2\binom{n-k-1}{k-1}}} = \frac{2\binom{n}{k}}{1+\frac{n-k}{k}} = 2\binom{n-1}{k-1} = |\mathcal{S}|.
	\end{align*}
	
\end{proof}

By Theorem~\ref{thm:cocliques-module}, we deduce the following corollary.
\begin{cor}
	The eigenvalue $-2\binom{n-k-1}{k-1}$ is afforded by the $[n-1,1]$-module.
\end{cor}

Therefore, the answer to Question~\ref{qst:main} is affirmative for association scheme corresponding to the Gelfand pair $(G,H)$, where $G = \sym(n)$ and $H = \alt(k) \times \sym(n-k)$.

In the next theorem, we find the maximum cocliques of $\mathcal{K}^\pm(n,k)$.
\begin{thm}
	The maximum cocliques of $\mathcal{K}^\pm (n,k)$ are of the form 
	\begin{align*}
		\left\{ A^\pm :\ A\in \binom{[n]}{k} \mbox{ and } x\in A \right\},
	\end{align*}
	for some $x \in [n]$.
\end{thm}
\begin{proof}
	 Let $\mathcal{S}$ be a maximum coclique of $\mathcal{K}^\pm (n,k)$. 
	 Let 
	 \begin{align*}
	 	\mbox{$\mathcal{S}_+ = \left\{ A^+ \in \mathcal{S}:\ A\in \binom{[n]}{k} \right\}$ and $\mathcal{S}_- = \left\{ A^- \in \mathcal{S}:\ A\in \binom{[n]}{k} \right\}$.}
	 \end{align*}
 	 Note that $|\mathcal{S}| = 2\binom{n-1}{k-1}$. Moreover, since $K^\pm (n,k) = K_2\bowtie K(n,k)$ and the maximum cocliques of $K(n,k)$ have size $\binom{n-1}{k-1}$, we have $|\mathcal{S}_+| = |\mathcal{S}_-|  = \binom{n-1}{k-1}$. That is, $\mathcal{S}_-$ and $\mathcal{S}_+$ are maximum cocliques in the two copies of $K(n,k)$ in $K^\pm(n,k)$. Therefore, there exists $a,b\in [n]$ such that
 	 \begin{align*}
 	 		\mbox{$\mathcal{S}_+ = \left\{ A^+ \in \mathcal{S}:\ A\in \binom{[n]}{k} \mbox{ and } a\in A \right\}$ and $\mathcal{S}_- = \left\{ A^- \in \mathcal{S}:\ A\in \binom{[n]}{k}  \mbox{ and } b\in A\right\}$.}
 	 \end{align*}
  	 It is clear that $\mathcal{S} = \mathcal{S}_- \cup \mathcal{S}_+$ is a coclique of $K^\pm(n,k)$ if and only if $a=b$. This completes the proof.
\end{proof}

\section{The orbital scheme from line~4}\label{sect:line4}

Consider the Gelfand pair $(G,H) = (\sym(n),\alt(k) \times \alt(n-k))$. First, we will determine certain combinatorial objects that correspond to $G/H$. Then, we will prove the main result.

By \cite{godsil2010multiplicity}, we first recall that
\begin{align*}
	\ind{H}{G} = \sum_{i=0}^k\left( \chi^{[n-i,i]} + \chi^{[2i,1^{n-2i}]} \right) + \chi^{[n-k,1^k]}+ \chi^{[k,1^{n-k}]} + \chi^{[n-k+1,1^{k-1}]} + \chi^{[k+1,1^{n-k-1}]}.
\end{align*}
Hence, $[n-1,1]$ is the second largest in dominance ordering in $\Lambda(n,H)$.
\subsection{Combinatorial objects}

For any $k$-subset $A$, let $\underline{A}$ be the complement of $A$ in $[n]$. Consider the set
\begin{align*}
	\Omega_{n,k}=\left\{ (A^\pm,\underline{A}^\pm) : \ A\in \binom{[n]}{k} \right\}.
\end{align*}
Note that $|\Omega| = 4\binom{n}{k}$.
We claim that the action of $G$ on $G/H$ is permutation equivalent to a certain action of $\sym(n)$ on $\Omega_{n,k}$. For any $\sigma\in \sym(n)$ and $(A^{*_1},\underline{A}^{*_2}) \in \Omega_{n,k}$, define 
\begin{align}
	\sigma\left( (A^{*_1},\underline{A}^{*_2})\right) := \left( \sigma(A^{*_1}),\sigma(\underline{A}^{*_2}) \right).
\end{align}
Note that this is an induced action from the action given in \eqref{eq:action}. This gives an action of $\sym(n)$ on $\Omega_{n,k}$. Let us now prove that the stabilizer of an element of $\Omega_{n,k}$ is conjugate to $H$. Consider the element $(A^+,\underline{A}^+) \in \Omega_{n,k}$, where $A = \{1,2,\ldots,k\}$. If $\sigma \in \sym(n)$ fixes $(A^+,\underline{A}^+)$, then we have
\begin{align*}
	\sigma(A^+)= A^+ \mbox{ and } \sigma(\underline{A}^+) = \underline{A}^+.
\end{align*}
In other words, $\sigma_{|A}$ and $\sigma_{|\underline{A}}$ must be even permutations. Consequently, the stabilizer of $(A^+,\underline{A}^+)$ is equal to $H$. One can also prove that $H$ is also the stabilizer of $(A^-,\underline{A}^-),\ (A^+,\underline{A}^-),\ \mbox{and } (A^-,\underline{A}^+)$.

We omit the proof of the following proposition since it straightforward.
\begin{prop}
	The action of $G$ on $G/H$ is permutation isomorphic to the action of $G$ on $\Omega_{n,k}$.
\end{prop}

\subsection{The orbital scheme $\mathcal{J}^{++}(n,k)$}
In this subsection, we determine the graphs in the orbital scheme corresponding to $(G,H)$.
\begin{thm}
	An orbital of $\sym(n)$ in its action on $\Omega_{n,k}$ is one of the following
	\begin{enumerate}[1)]
		\item $\mathcal{O}_0^{+,+} := \left\{\left((A^+,\underline{A}^+),(A^+,\underline{A}^+)\right):\ A\in  \binom{[n]}{k} \right\}$
		\item $\mathcal{O}_0^{-,-}:= \left\{ \left((A^+,\underline{A}^+),(A^-,\underline{A}^-)\right) :\ A\in \binom{[n]}{k} \right\}$
		\item $\mathcal{O}_0^{+,-} := \left\{\left((A^+,\underline{A}^+),(A^+,\underline{A}^-)\right):\ A\in  \binom{[n]}{k} \right\}$
		\item $\mathcal{O}_0^{-,+}:= \left\{ \left((A^+,\underline{A}^+),(A^-,\underline{A}^+)\right) :\ A\in \binom{[n]}{k} \right\}$
		\item $\mathcal{O}_1^{+,+} := \left\{\left((A^+,\underline{A}^+),(B^+,\underline{B}^+)\right):\ A,B\in  \binom{[n]}{k} \mbox{ and } |A\cap B| = k-1 \right\}$
		\item $\mathcal{O}_1^{-,-}:= \left\{ \left((A^+,\underline{A}^+),(B^-,\underline{B}^-)\right) :\ A,B\in \binom{[n]}{k} \mbox{ and } |A \cap B| = k-1 \right\}$
		\item $\mathcal{O}_1^{+,-} := \left\{\left((A^+,\underline{A}^+),(B^+,\underline{B}^-)\right):\ A,B\in  \binom{[n]}{k} \mbox{ and } |A\cap B| = k-1 \right\}$
		\item $\mathcal{O}_1^{-,+}:= \left\{ \left((A^+,\underline{A}^+),(B^-,\underline{B}^+)\right) :\ A,B\in \binom{[n]}{k} \mbox{ and } |A\cap B| = k-1 \right\}$
		\item  for $2\leq i\leq  k$ 
		\begin{align*}
			\mathcal{O}_i^+ := \left\{ \left((A^+,\underline{A}^+),(B^+,\underline{B}^+)\right),\left((A^+,\underline{A}^+),(B^-,\underline{B}^-)\right) : A,B\in \binom{[n]}{k} \mbox{ and } |A\cap B| = k-i \right\},
		\end{align*}
		\item for $2\leq i \leq k$, 
		\begin{align*}
			\mathcal{O}^-_i := \left\{ \left((A^+,\underline{A}^+),(B^+,\underline{B}^-)\right),
			 \left((A^+,\underline{A}^+),(B^-,\underline{B}^+)\right): A,B\in \binom{[n]}{k} \mbox{ and } |A\cap B| = k-i \right\}.
		\end{align*}
	\end{enumerate}
\end{thm}
\begin{proof}
	Let $x_1<x_2<\ldots<x_k$ be distinct elements of $[n] \setminus [k]$. Define
	\begin{align*}
		X_i = \left\{ 1,2,\ldots,k-i,x_1,x_2,\ldots,x_i \right\}, \ \mbox{ for } 0\leq i\leq k.
	\end{align*}
	Note that $X_0 = \{1,2\ldots,k\}$ and $X_k = \{x_1,x_2,\ldots,x_k\}$. It is enough to prove that the set 
	\begin{align*}
		S &=\left\{ (X_i^\pm,\underline{X}_i^\pm): 0\leq i \leq 1\right\} \cup \left\{ (X_i^+,\underline{X}_i^+) :\ 2\leq i\leq k \right\} \cup \left\{ (X_i^+,\underline{X}_i^-):\ 2 \leq i \leq k \right\},
	\end{align*}
	consists  of elements from distinct  suborbits of $(X_0^+,\underline{X}_0^+)$. 
	Recall that the stabilizer of $(X_0^+,\underline{X}_0^+)$ is $\alt(k) \times \alt(n-k)$. Similar to what we saw in the previous section, it is clear that an element of $\alt(k) \times \alt(n-k)$ cannot map $(X_i^\pm,\underline{X}_i^\pm)$ to $(X_j^\pm,\underline{X}_j^\pm)$, unless $i = j$. 
	
	Since $\alt(k) \times \alt(n-k)$ stabilizes $(X^+_0,\underline{X}_0^+)$ and $\left(X_0^-,\underline{X}_0^-\right)$, these two cannot be in the same suborbit. If $\sigma \in \sym(n)$ such that $\sigma (X_0^+,\underline{X}_0^-) = (X_0^-,\underline{X}_0^+)$, then it is straightforward that $\sigma \not\in\alt(k) \times \alt(n-k)$. Hence, the four elements $(X_0^\pm,\underline{X}_0^\pm)$ are in different suborbits.
	
	Assume now that there exists $\sigma \in \alt(k) \times \alt(n-k)$ such that $(X_1^+,\underline{X}_1^+)$  is mapped to $(X_1^- ,\underline{X}_1^-)$ by $\sigma$. It is necessary that the set $X_1$ is fixed by $\sigma$ setwise. It is clear that $\sigma(k) = k$ and $\sigma(x_1) = x_1$ since $\sigma$ also fixes $X_0 = \{1,2,\ldots,k\}$ and $\underline{X}_0$. We conclude that
	\begin{align*}
		\sigma_{|X_1} = \sigma_{|X_0 \setminus \{ k \}} \mbox{ is an even permutation}.
	\end{align*}
	Consequently, $\sigma$ cannot map $X_0^+$ to $X_0^-$. The other cases can be proved in the same way. Consequently, the elements $(X_1^\pm,\underline{X}_1^\pm)$ are in different suborbits of $\alt(k) \times \sym(n-k)$.
	
	Finally, we assume that $2\leq i\leq k$. We prove that $(X_i^+,\underline{X}^+_i)$ and $(X_i^+,\underline{X}_i^-)$ are not in the same suborbits by contradiction. Assume that $\sigma \in \alt(k) \times \alt(n-k)$ maps $(X_i^+,\underline{X}_i^+)$ to $(X_i^+,\underline{X}_i^-)$. Again, we note that $\sigma$ fixes $X_i$ and $\underline{X}_i$ setwise. Combining this with the fact that $\sigma$ permutes the elements of $X_0$ and $\sigma_{|X_i}$ is an even permutation, we deduce that
	\begin{align*}
		\sigma_{|\{1,2\ldots,k-i\}} \mbox{ and } \sigma_{|\{x_1,x_2,\ldots,x_i\}}
	\end{align*}
	must be both even or both odd permutations. Without loss of generality, assume that they are both even permutations. Noting that $\sigma_{|\underline{X}_i}$ is an odd permutation (it maps $\underline{X}_i^+$ to $\underline{X}_i^-$), by a similar argument as above, we deduce that one of
	\begin{align*}
		\sigma_{|\{k-i+1,\ldots,k\}} \mbox{ and } \sigma_{|\underline{X}_0 \setminus \{x_1,x_2,\ldots,x_i\} }
	\end{align*}
	must be even and the other one is an odd permutation. As $\sigma_{|\{1,2\ldots,k-i\}}$ is an even permutation and $\sigma(X_0^+) = X^+_0$, we conclude that $\sigma_{|\{k-i+1,\ldots,k\}}$ must be an even permutation and $\sigma_{|\underline{X}_0 \setminus \{x_1,x_2,\ldots,x_i\} }$ is an odd permutation. In summary, $\sigma_{|\{x_1,x_2,\ldots,x_i\}}$ is an even permutation and $\sigma_{|\underline{X_{0}}\setminus \{x_1,x_2,\ldots,x_i\}}$ is an odd permutation. This implies that $\sigma_{|\underline{X}_0}$ is an odd permutation. The latter is impossible. Therefore, the two elements that we started with cannot be in the same orbital.
	
	Since the corresponding action of $G$ has rank $2k+6$, the set $S$ consists of an element from each suborbits. It is easy to check that each element of $S$ is a representative of the ten sets in the statement of the theorem. The rest of the proof follows by making the symmetric group $\sym(n)$ act on the pairs $\left((X_0^+,\underline{X}_0^+),T\right)$, where $T \in S$.
\end{proof}

We will denote the association scheme given by the orbitals in the previous theorem by $\mathcal{J}^{++}(n,k)$. It is not hard to check that all orbitals listed above are self-paired. Let $\mathcal{K}(n,k)$ be the graph corresponding to the orbital $\mathcal{O}^{+}_k$.

\begin{prop}
	The graph $\mathcal{K}(n,k) $ is isomorphic to two disjoint copies of $ K_2\bowtie K(n,k)$.
\end{prop}
\begin{proof}
	Partition the vertices of $\mathcal{K}(n,k)$ into $\{ V^{+,+}, V^{+,-},V^{-,+},V^{-,-} \}$, where
	\begin{align*}
		V^{+,+} &= \left\{ (A^+,\underline{A}^+) :\ A\in \binom{[n]}{k} \right\},\\
		V^{+,-} &= \left\{ (A^+,\underline{A}^-) :\ A\in \binom{[n]}{k} \right\},\\
		V^{-,+} &= \left\{ (A^-,\underline{A}^+) :\ A\in \binom{[n]}{k} \right\},\\
		V^{-,-} &= \left\{ (A^-,\underline{A}^-) :\ A\in \binom{[n]}{k} \right\}.
	\end{align*}
	By definition of $\mathcal{O}^+_k$, it is clear that there is no edge between a vertex in $V^{+,+} \cup V^{-,-}$ and a vertex in $V^{+,-} \cup V^{-,+}$. In other words, $\mathcal{K}(n,k)$ is disconnected. A vertex $(A^+,\underline{A}^+)$ is adjacent to all vertices of the form $(B^+,\underline{B}^+)$ and $(B^-,\underline{B}^-)$, for all $k$-subsets $B$ with the property that $A\cap B = \varnothing$. This proves that the subgraph of $\mathcal{K}(n,k)$ induced by $V^{+,+} \cup V^{-,-}$ is in fact isomorphic to $K_2 \bowtie K(n,k)$. One can also prove with the same argument that the subgraph induced by $V^{-,+} \cup V^{+,-}$ is also isomorphic to $K_2 \bowtie K(n,k)$.
\end{proof}
\subsection{The $[n-1,1]$-module}

By Lemma~\ref{lem:two-orbits}, $K = \sym([n-1,1])$ has two orbits on $\Omega_{n,k}$. It is immediate that the partition from these orbits is $\pi = \{ \mathcal{S}, \Omega\setminus \mathcal{S} \}$, where
\begin{align*}
	\mathcal{S} = \left\{ \left(A^\pm,\underline{A}^\pm\right) :\ A\in \binom{[n]}{k} \mbox{ and } n\in A \right\}.
\end{align*}
It is clear that $\mathcal{S}$ is a coclique of $\mathcal{K}(n,k)$ since for any $(A^\pm,\underline{A}^\pm),(B^\pm,\underline{B}^\pm) \in \mathcal{S}$, $|A\cap B|\geq 1$. We claim that $\mathcal{S}$ is also a maximum coclique of $\mathcal{K}(n,k)$. To do this, we use the Ratio Bound as follows.

The proof of the following proposition is identical to the proof of Corollary~\ref{cor:eigenvalues-quasi-kneser}.
\begin{prop}
	The eigenvalues of $\mathcal{K}(n,k)$ are either $0$ or 
	\begin{align*}
		2(-1)^j\binom{n-k-j}{k-j},
	\end{align*}
	for $0\leq j\leq k$.
\end{prop}

From this proposition, we deduce that the smallest eigenvalue of $\mathcal{K}(n,k)$ is $-2\binom{n-k-1}{k-1}$ and the largest eigenvalue is $2\binom{n-k}{k}$. By the Ratio Bound, we have
\begin{align*}
	\alpha(\mathcal{K}(n,k)) \leq \frac{4\binom{n}{4}}{1-\frac{2\binom{n-k}{k}}{-2\binom{n-k-1}{k-1}}} = 4 \binom{n-1}{k-1}.
\end{align*}
Consequently, $\mathcal{S}$ is a coclique of maximum size.
By Theorem~\ref{thm:cocliques-module}, we deduce that the least eigenvalue $-2\binom{n-k-1}{k-1}$ is the eigenvalue afforded by the $[n-1,1]$-module.
\begin{rmk}
	In this subsection, we proved that the $[n-1,1]$-module is a subspace of the eigenspace corresponding to the eigenvalue $-2\binom{n-k-1}{k-1}$. However, this eigenspace is much larger than the $[n-1,1]$-module since the graph is a disjoint union of two copies of the same graph. We conjecture that the eigenvalue afforded by the $[2,1^{n-2}]$-module is also equal to $-2\binom{n-k-1}{k-1}$. This conjecture will be true if there exists $x\in\sym(n)$ such that $xH \subset \alt(n)$ and $(H,xH)$ corresponds to an edge in $\mathcal{K}(n,k)$. To see this, we first note that
	\begin{align*}
		\sum_{h\in H}\chi^{[n-1,1]}(xh) = -\frac{|H|}{\binom{n-k}{k}}\binom{n-k-1}{k-1}.
	\end{align*}
	Using this equality, the eigenvalue corresponding to $[2,1^{n-2}]$-module is
	\begin{align*}
		\frac{2\binom{n-k}{k}}{|H|}\sum_{h\in H} \chi^{\left[2,1^{n-2}\right]}(xh) = \frac{2\binom{n-k}{k}}{|H|} \sum_{h\in H} \chi^{[1^n]}(xh) \chi^{[n-1,1]}(xh) = -2\binom{n-k-1}{k-1}.
	\end{align*}
\end{rmk}

\section{The orbital scheme from line~5}\label{sect:line5}

Let $G=\sym(n)$, $H = \left(\sym(k) \times \sym(n-k)\right)\cap \alt(n)$ and consider the Gelfand pair $(G,H)$. Similar to the previous cases, we will introduce certain combinatorial objects that correspond to $G/H$. Then, we will find an orbital graph in the orbital scheme that gives an affirmative answer to Question~\ref{qst:main}.
\subsection{Combinatorial objects}
Recall that $\mathcal{I}_{n,k}$ is the set of all injections from $[k]$ to $[n]$. Given $A\in \mathcal{I}_{n,k}$, let $\operatorname{Im}(A)$ be the image of the map $A$ and define 
\begin{align}
	\underline{\mathcal{I}}_{n,k}(A):=\left\{ g: [n-k] \to[n]\setminus \operatorname{Im}(A) :\ g\mbox{ is injective} \right\}.
\end{align}
Since we may view an element of $\mathcal{I}_{n,k}$ as a sequence of the form $A = (a_1,a_2,\ldots,a_k)$ with distinct entries, recall that we have an action of $\sym(k)$ on $\mathcal{I}_{n,k}$ through $\sigma(A) := \left(a_{\sigma(1)},a_{\sigma(2)},\ldots,a_{\sigma(k)}\right)$, for any $\sigma\in \sym(k)$. Similarly, the group $\sym(n-k)$ also acts on any element of $\mathcal{I}_{n,n-k}$ by permuting the indices. These two actions give an induced action of $\sym(k)\times \sym(n-k)$ on $\{A\} \times \underline{\mathcal{I}}_{n,k}(A)$, for any $A\in\mathcal{I}_{n,k}$. Now, define $\mathcal{T}_{n,k} := \left\{ \left(A,B\right):\ A\in \mathcal{I}_{n,k} \mbox{ and } B \in \underline{\mathcal{I}}_{n,k}(A) \right\}$ and consider the relation $\mathcal{R}$ on $\mathcal{T}_{n,k}$ such that $(A,B)\mathcal{R}(C,D)$ if and only if there exists $\sigma\in \alt(n)$ such that
\begin{align*}
 	 \sigma(A) = C \mbox{ and } \sigma(B) = D.
\end{align*}
Let us compute the number of equivalence classes of $\mathcal{R}$. For any $(A,B)\in \mathcal{T}_{n,k}$, the equivalence class of $\mathcal{R}$ that contains $(A,B)$ is
\begin{align}
	\overline{(A,B)} :=\left\{ (\sigma(A),\sigma(B)):\ \sigma \in \left(\sym(k) \times\sym(n-k)\right) \cap \alt(n) \right\}
\end{align}

Let $S = \{a_1,a_2,\ldots,a_k\}$ and $\underline{S}:=[n]\setminus S = \{b_1,b_2,\ldots,b_{n-k}\}$. Consider the set $T = \left\{(A,B) \in \mathcal{T}_{n,k} :\ \operatorname{Im}(A) = S \right\}$. The group $\left(\sym(k) \times \sym(n-k)\right) \cap \alt(n)$ acts intransitively on $T$ since the orbit counting lemma gives
{\footnotesize
\begin{align*}
	\frac{2}{k!(n-k)!} \sum_{\sigma\in \left(\sym(k) \times \sym(n-k)\right) \cap \alt(n)} \left|\left\{ (A,B) \in T :\ \sigma.A = A \mbox{ and } \sigma.B = B \right\}\right| = \frac{2}{k!(n-k)!} k!(n-k)! = 2.
\end{align*}
}
Therefore, the set $S$ determines two orbits which are
\begin{align*}
	(S,\underline{S})^+ &:=  \overline{((a_1,a_2,\ldots,a_k),(b_1,b_2,\ldots,b_{n-k}))} ,\\
	(S,\underline{S})^- &:= \overline{((a_2,a_1,\ldots,a_k),(b_1,b_2,\ldots,b_{n-k}))}.
\end{align*}
Note that $(S,\underline{S})^+$ and $(S,\underline{S})^-$ are both fixed by $\left(\sym(k) \times \sym(n-k)\right)\cap \alt(n)$ and they are swapped by any odd permutation in $\left(\sym(k) \times \sym(n-k)\right)$.

For the remainder of this section, we let 
\begin{align*}
	\Omega_{n,k} := \left\{ (S,\underline{S})^\pm :\ S\in \binom{[n]}{k} \right\}.
\end{align*}
\subsection{Action of the symmetric group}
For any $A \in \mathcal{I}_{n,k}$, define
\begin{align}
	\sgn{\sigma(A),\sigma(\underline{A})} &:= 
	\begin{cases}
		+ & \mbox{ if }\sgn{\mathsf{sort}(\sigma(A))}\times\sgn{\mathsf{sort}(\sigma(\underline{A}))} = 1,\\
		- & \mbox{ otherwise.}
	\end{cases}
\end{align}
Recall that for any $k$-subset $S$ of $[n]$, $\hat{S}$ is the tuple obtained by arranging the entries of $S$ in an increasing order. For any $\sigma \in \sym(n)$ and $S\in \binom{[n]}{k}$, define
\begin{align}
	\sigma\left((S,\underline{S})^*\right)
	:=
	\begin{cases}
		(\sigma(S),\sigma(\underline{S}))^{\mathsf{sgn}(\sigma(\hat{S}),\sigma(\underline{\hat{S}}))} & \mbox{ if } * = +,\\
		(\sigma(S),\sigma(\underline{S}))^{-\mathsf{sgn}(\sigma(\hat{S}),\sigma(\underline{\hat{S}}))} & \mbox{ if } * = -.
	\end{cases}
\end{align}
It is clear that $\sym(n)$ acts transitively on $\Omega_{n,k}$ with this action. Moreover, for any $S\in \binom{[n]}{k}$, a permutation $\sigma$ is in the stabilizer of $(S,\underline{S})^+$ if $\sigma $ fixes $S$ setwise (which also implies that it fixes $\underline{S}$ setwise), and  $\sgn{\sigma(\hat{S}),\sigma(\underline{\hat{S}})} = +$ which is equivalent to $\sigma_{|S}$ and $\sigma_{|\underline{S}}$ being both even permutations or both odd permutations. Hence,
\begin{align*}
	\stab{\sym(n)}{(S,\underline{S})^+} = \left(\sym(k) \times \sym(n-k)\right)\cap \alt(n).
\end{align*}
The permutation $\sigma$ is in the stabilizer of $(S,\underline{S})^-$ if $\sigma$  fixes both $S$ and $\underline{S}$ setwise, and $\sgn{\sigma(S),\sigma(\underline{S})} = -$ which is equivalent to saying that $\sigma_{|S}$ and $\sigma_{|\underline{S}}$ are both even or odd permutations. 
Consequently, 
\begin{align*}
	\stab{\sym(n)}{(S,\underline{S})^-} = \left(\sym(k) \times \sym(n-k)\right) \cap \alt(n).
\end{align*}
The following proposition can be easily verified.
\begin{prop}
	The action of $\sym(n)$ on the cosets of $\left(\sym(k)\times \sym(n-k)\right)\cap \alt(n)$ is permutation equivalent to the action of $\sym(n)$ on $\Omega_{n,k}$.
\end{prop}
\subsection{The orbital scheme $\mathcal{J}^\pm(n,k)$}

We will denote the orbital scheme obtained from $(G,H) = (\sym(n),\left(\sym(k)\times \sym(n-k)\right) \cap \alt(n))$ by $\mathcal{J}^\pm (n,k)$.
\begin{thm}
	An orbital of $\sym(n)$ acting on the $\Omega_{n,k}$ is one of the following
	\begin{enumerate}[1)]
		\item $\mathcal{O}_0^+ := \left\{\left((A,\underline{A})^+,(A,\underline{A})^+\right) :\ A\in \binom{[n]}{k} \right\}$,
		\item $\mathcal{O}_0^-:=\left\{\left((A,\underline{A})^+,(A,\underline{A})^-\right):\ A\in \binom{[n]}{k}\right\}$,
		\item $\mathcal{O}_i^-:=\left\{ \left((A,\underline{A})^+,(B,\underline{B})^-\right)  :\ |A\cap B| = k-i \right\}$, for $1\leq i\leq k$
		\item $\mathcal{O}_i^+:=\left\{ \left((A,\underline{A})^+,(B,\underline{B})^+\right) :\ |A\cap B| = k-i \right\}$, for $1\leq i\leq k$. 
	\end{enumerate}\label{thm:orbitals1}
\end{thm}
\begin{proof}
	We note first that the action of $\sym(n)$ on $\Omega_{n,k}$ has rank $2k+2$ (see \cite{godsil2010multiplicity} for details).		
	Consider the $k$ distinct elements $x_1<x_2<\ldots<x_k$ of $[n]\setminus \{1,2,\ldots,k\}$. Let $X_i = \{1,2,\ldots,k-i,x_1,x_2,\ldots,x_i\}$ for $0\leq i\leq k$. Note that $X_0 = \{1,2,\ldots,k\},\ X_k = \{ x_1,x_2,\ldots,x_k \}$. We claim that the set 
	\begin{align*}
		S = \left\{ (X_i,\underline{X_i})^+ :\ 0\leq i \leq k \right\} \cup \left\{ (X_i,\underline{X_i})^-:\ 0\leq i\leq k  \right\}
	\end{align*}
	has $2k+2$ elements in different suborbits of $H = \left(\sym(k)\times \sym(n-k)\right) \cap \alt(k)$. 
	
	First, note that if $A \in \binom{[n]}{k}$ such that $|A\cap X| = i$, then $((X,\underline{X})^+,(A,\underline{A})^*)$ can only be mapped by $H$ to an element of the form $((X,\underline{X})^+,(B,\underline{B})^\dagger)$, where $|B\cap X| = i$.
	Now, if $\sigma \in \sym(n)$ such that $\sigma\left((X_i,\underline{X_i})^+\right) = (X_i,\underline{X_i})^-$, then by definition we must have $\sgn{\sigma(\hat{X_i}),\sigma(\hat{\underline{X}_i})}= -$. Hence, we can assume without loss of generality that $\sigma_{|X_i}$ is an even permutation and $\sigma_{|\underline{X_i}}$ is an odd permutation. Therefore, $\sigma_{|X_0\cap X_i}$ and $\sigma_{|X_0\cap \underline{X}_i}$ must have the same sign, whereas $\sigma_{|\underline{X}_0\cap {X}_i}$ and $\sigma_{|\underline{X}_0\cap \underline{X}_i}$ have opposite signs. Using this, one can easily prove that $\sigma\left((X_0,\underline{X}_0)\right)^+$ cannot be equal to $(X_0,\underline{X}_0)^+$. This completes the proof.
\end{proof}

Let $\mathcal{K}^+(n,k)$ and $\mathcal{K}^-(n,k)$ be the two orbital graphs corresponding to $\mathcal{O}_k^+$ and $\mathcal{O}_k^-$, respectively. By definition of the orbital $\mathcal{O}^\pm_k$, the following lemma follows immediately.
\begin{lem}
	We have $\mathcal{K}^-(n,k) = K_2\times K(n,k)$ and $\mathcal{K}^+(n,k)$ is a disjoint union of two copies of $K(n,k)$.
\end{lem}
\begin{proof}
	In $\mathcal{K}^-(n,k)$, the edges are of the form $\left((A,\underline{A})^+,(B,\underline{B})^-\right)$ or $\left((A,\underline{A})^-,(B,\underline{B})^+\right)$, where $|A\cap B| = k-1$. Since this graph is bipartite, we clearly have $\mathcal{K}^-(n,k) = K_2\times K(n,k)$. 
	
	The graph $\mathcal{K}^+(n,k)$ is disconnected since there is no edge between vertices of the form $(A,\underline{A})^+$ and $(B,\underline{B})^-$, for any $A,B\in \binom{[n]}{k}$. The subgraph of $\mathcal{K}^-(n,k)$ induced by all vertices of the form $(A,\underline{A})^+$, for all $A \in \binom{[n]}{k}$, is isomorphic to $K(n,k)$. The same holds for the vertices of the form $(A,\underline{A})^-$, for all $A\in \binom{[n]}{k}$. This completes the proof.
\end{proof}
\subsection{The $[n-1,1]$-module}

Let $K = \sym([n-1,1])$. The orbit partition induced by $K$ on $\Omega_{n,k}$ is $\pi = \{ \mathcal{S},\Omega_{n,k}\setminus \mathcal{S} \}$, where 
\begin{align*}
	\mathcal{S} = \left\{ \left(A,\underline{A}\right)^\pm :\ A\in \binom{[n]}{k} \mbox{ and } n\in A \right\}.
\end{align*}
We note that $|\mathcal{S}| = 2\binom{n-1}{k-1}$.
It is clear that $\mathcal{S}$ is a coclique of both orbital graphs of $\mathcal{K}^-(n,k)$ and $\mathcal{K}^+(n,k)$, however, it is not a maximum coclique of the one corresponding to $\mathcal{K}^-(n,k)$ since the latter is bipartite. That is, the graph $\mathcal{K}^-(n,k)$ gives an affirmative answer to the first part of Question~\ref{qst:main}, however, the second part of the question is not satisfied.
\begin{thm}
	The least eigenvalue of $\mathcal{K}^-(n,k)$ is afforded by the $[1^{n}]$-module.
\end{thm}
\begin{proof}
	By Theorem~\ref{thm:cocliques-module}, the eigenvalue afforded by the $[n-1,1]$-module is 
	\begin{align*}
		 -\frac{2\binom{n-k}{k}\binom{n-1}{k-1}}{2\binom{n}{k} - 2\binom{n-1}{k-1}}.
	\end{align*}
	Since $\mathcal{K}^-(n,k)$ is bipartite, its smallest eigenvalue is equal to $-\binom{n-k}{k}$. One can verify that this eigenvalue is afforded by the $[1^n]$-module by using Theorem~\ref{thm:general-eigenvalues}.
\end{proof}

\begin{thm}
	The least eigenvalue of $\mathcal{K}^+(n,k)$ is afforded by the $[n-1,1]$-module.
\end{thm}
\begin{proof}
	Since $\mathcal{S}$ is a coclique, by Theorem~\ref{thm:cocliques-module}, the eigenvalue $-2\frac{\binom{n-k}{k} \binom{n-1}{k-1}}{2\binom{n-1}{k}}$ is the eigenvalue afforded by the $[n-1,1]$-module. As $\mathcal{K}^+(n,k) $ is a disjoint union of two copies of $K(n,k)$,  its smallest eigenvalue is $-\binom{n-k-1}{k-1} = -\frac{\binom{n-k}{k} \binom{n-1}{k-1}}{\binom{n-1}{k}}$. 
\end{proof}

\section{The orbital scheme from line~6 of Table~\ref{table:intrans2}}\label{sect:qpm}

In this section, we prove that the answer to Question~\ref{qst:main} is affirmative for the group in line~6 of Table~\ref{table:intrans2}. First, we prove that Question~\ref{qst:main} is affirmative for the Gelfand pair $(\sym(2k),\sym(2)\wr \sym(k))$. Then, we use a graph isomorphism to prove that Question~\ref{qst:main} also holds for the Gelfand pair $(\sym(2k+1),\sym(2)\wr \sym(k))$. 

\subsection{Question~\ref{qst:main} for the Gelfand pair $(\sym(2k),\sym(2)\wr \sym(k)$}

For any $k\geq 2$, the action of $\sym(2k)$ on the cosets of $\sym(2) \wr \sym(k)$ is multiplicity-free (see \cite{godsil2010multiplicity}). We note that the second largest in dominance ordering in $\Lambda(2k,\sym(2)\wr \sym(k))$ is $[2k-2,2]$. The action of $\sym(2k)$ on the cosets of $\sym(2)\wr \sym(k)$ is equivalent to the action of $\sym(2k)$ on the perfect matchings of the complete graph $K_{2k}$. A perfect matching of $K_{2k}$ is a partition of the set $[2k]$ into $k$ subsets of size $2$. We will denote the set of all perfect matchings of $K_{2k}$ by $\mathcal{P}_k.$ 

For any $\lambda = [\lambda_1,\lambda_2,\ldots,\lambda_t]$, let $\mathcal{O}_{[2\lambda_1,2\lambda_2,\ldots,2\lambda_t]}$ be the set of all pairs $(P,Q)$ of elements of $\mathcal{P}_k$ such that $P \cup Q$ is a union of $t$ cycles of length $2\lambda_1,2\lambda_2,\ldots,2\lambda_t$. We note that an edge is considered a $2$-cycle in this definition.

For any $\lambda = [\lambda_1,\lambda_2,\ldots,\lambda_k] \vdash k$, we define $2\lambda := [2\lambda_1,2\lambda_2,\ldots,2\lambda_k]$. The association scheme arising from $(\sym(2k),\sym(2)\wr \sym(k))$ is the well-known \emph{perfect matching association scheme}. 
The orbitals of the corresponding group action are all $\mathcal{O}_{2\lambda}$, where $\lambda \vdash k$. Next, we prove that the answer to Question~\ref{qst:main} follows from a result in \cite{godsil2016algebraic,lindzey2017erdHos}.

For any $n$, we let 
\begin{align}
	n!!=
	\begin{cases}
		n\times (n-2) \times  \ldots \times 3 \times 1 & \mbox{ if $n$ is odd} \\
		n \times (n-2) \times \ldots \times 4 \times 2 & \mbox{ otherwise.}
	\end{cases}
\end{align} 
Using this notation, it is not hard to see that
\begin{align*}
	|\mathcal{P}_k| &= (2k-1)!!.
\end{align*}

For any $\mathcal{F} \subset \mathcal{P}_k$, we say that $\mathcal{F}$ is intersecting if $|P \cap Q| \geq 1$, for any $P$ and $Q$ in $\mathcal{F}$. We state the following theorem about intersecting families of $\mathcal{P}_k$.
\begin{thm}[\cite{godsil2016algebraic,lindzey2017erdHos}]
	If $\mathcal{F} \subset \mathcal{P}_k$ is intersecting, then $|\mathcal{F}|\leq (2k-3)!!$. Moreover, equality holds if and only if $\mathcal{F}$ consists of all perfect matchings with a fixed edge, i.e., an orbit of size $(2k-3)!!$ of a conjugate of $\sym(2k-2) \times \sym(2)$.\label{thm:pm}
\end{thm}

To prove Theorem~\ref{thm:pm}, the authors relied on cocliques in the so-called \emph{perfect matching derangement graph} $P(k)$. The vertices of this graph consists the elements of $\mathcal{P}_k$ and two perfect matchings are adjacent if they are not intersecting. The graph $P(k)$ is the union of all orbital graphs corresponding to $\mathcal{O}_{2\lambda}$, where $\lambda \vdash k$ and $2\lambda$ does not contain any part of size $2$. It follows from Theorem~\ref{thm:pm} that any orbital graph which is a subgraph of $P(k)$ gives an affirmative answer to Question~\ref{qst:main}~\eqref{qst:part1}. Using the upper bound on the dimension of the Specht modules in $\Lambda(2k,\sym(2)\wr \sym(k))$ given in \cite[Lemma~3.7]{behajaina2023intersection} (or \cite[Lemma~3.2]{behajaina20213}) and the ``trace trick'' as given in \cite[Theorem~7.2]{godsil2016algebraic}, one can prove that the smallest eigenvalue of the orbital graph corresponding to the partition $[2k]$ is equal to $-(2k-2)!!$. The latter is afforded by the $[2k-2,2]$-module. Thus Question~\ref{qst:main}~\eqref{qst:part2} is affirmative.

\subsection{Question~\ref{qst:main} for the Gelfand pair $(\sym(2k+1),\sym(2)\wr \sym(k))$}

Let $k$ be a positive integer. For any $\lambda = [\lambda_1,\lambda_2,\ldots,\lambda_t] \vdash k$, under the assumption that $\lambda_{0} = k$, we let 
\begin{align*}
	\mathcal{I}(\lambda):= \left\{ i\in \{1,2,\ldots,t\} \mid \lambda_{i-1}>\lambda_{i} \right\}.
\end{align*}
For any $\lambda = [\lambda_1,\lambda_2,\ldots,\lambda_t] \vdash k$, define the partitions of $2k+1$ given by
\begin{align*}
	\begin{cases}
		2\lambda^{(i)}+1 &:= [2\lambda_1,2\lambda_2,\ldots , 2\lambda_i+1,\ldots,2\lambda_t], \ \mbox{ for } i\in \mathcal{I}(\lambda)\\
		2\lambda^{(t+1)}+1 &:= [2\lambda_1,2\lambda_2,\ldots,2\lambda_t,1].
	\end{cases}
\end{align*}
Further, for any $\lambda = [\lambda_1,\lambda_2,\ldots,\lambda_t] \vdash k$ we define the set
\begin{align*}
	2\lambda +1 = \left\{ 2\lambda^{(i)}+1 : i\in \mathcal{I}(\lambda) \cup \left\{t+1\right\} \right\}.
\end{align*}
Finally, define the set
\begin{align*}
	\Lambda_k := \bigcup_{\lambda \vdash k} 2\lambda +1 .
\end{align*}

The group $\sym(2) \wr \sym(k)$ is a multiplicity-free subgroup of $\sym(2k+1)$ \cite{godsil2010multiplicity}. Using \cite{godsil2010multiplicity}, we know that
\begin{align*}
	\ind{\sym(2) \wr \sym(k)}{\sym(2k+1)} = \sum_{\mu \in \Lambda_k} \chi^{\mu} .
\end{align*}

The action of $\sym(2k+1)$ on cosets of $\sym(2) \wr \sym(k)$ is equivalent to its action on the set $\mathcal{Q}_k$ of all partitions of $[2k+1]$ into $k$ sets of size two and a singleton. Such a partition in $\mathcal{Q}_k$ is called a \emph{quasi-perfect matching} of $K_{2k+1}$. It is clear that
\begin{align*}
	|\mathcal{Q}_k| = (2k+1)!!.
\end{align*}

In the next lemma, we determine the orbitals of $\sym(2k+1)$ in its action on quasi-perfect matchings of $K_{2k+1}$. We omit the proof of this lemma since it is similar to how the orbital graphs of the action of $\sym(2k)$ on cosets of $\sym(2) \wr \sym(k)$ are obtained (see \cite{godsil2016erdos} for details).
\begin{lem}
	An orbital of $\sym(2k+1)$ in its action on quasi-perfect matchings of $K_{2k+1}$ is of the form:
	\begin{enumerate}[(i)]
		\item $\mathcal{O}_{ [2\lambda_1,\ldots,2\lambda_i+1, \ldots,2\lambda_t] } $, for some partition $[\lambda_1,\lambda_2,\ldots,\lambda_t] \vdash k$, where $(P,P^\prime) \in \mathcal{O}_{ [2\lambda_1,\ldots,2\lambda_i+1, \ldots,2\lambda_t] }$ if and only if $P\cup P^\prime$ is a disjoint union of $t-1$ cycles of length $2\lambda_1,\ldots ,2 \lambda_{i-1},2 \lambda_{i+1},\ldots, 2\lambda_t$ and a path of length $2\lambda_i+1$.
		\item $\mathcal{O}_{[2\lambda_1,\ldots,2\lambda_t,1]}$, for some partition $[\lambda_1,\ldots,\lambda_t] \vdash k$, where $(P,P^\prime) \in \mathcal{O}_{[2\lambda_1,\ldots,2\lambda_t,1]}$ if and only if $P\cup P^\prime$ is a disjoint union of $t-1$ cycles of length $2\lambda_1,\ldots,2\lambda_t$ and an isolated vertex.
	\end{enumerate}
\end{lem}

It is clear that the orbital scheme arising from the action of $\sym(2k+1)$ on the quasi-perfect matchings of $K_{2k+1}$ consists of undirected graphs. 

For any $k\geq 2$, let $Q(k)$ be the graph consisting of the union of all orbital graphs corresponding to $\mathcal{O}_{\mu}$, such that $\mu \in \Lambda_k$ does not contain any part of size $1$ or $2$. We prove the next result about the relation between $Q(k)$ and the perfect matching derangement graph.
\begin{lem}
	For any $k\geq 2$, there exists graph isomorphism $\varphi$ from $Q(k)$ to $P(k+1)$.
\end{lem}
\begin{proof}
	Consider the map $\varphi: \mathcal{Q}_k \to \mathcal{P}_{k+1}$ such that any $Q \in \mathcal{Q}_k$ is mapped to the element $\varphi(Q)$ of $\mathcal{P}_{k+1}$  obtained from $Q$ by replacing the singleton $\{a\} \in Q$ by $\{a,2k+2\}$. It is clear $\varphi(Q)  $$ \in \mathcal{P}_{k+1}$ and $\varphi$ is well defined. By uniqueness of the singleton in any element $\mathcal{Q}_k$, it is clear that $\varphi$ is injective. The surjectivity is obtained by replacing the $2$-subset $\{a,2k+2\}$ in any element of $\mathcal{P}_{k+1}$ by $\{a\}$. Hence, $\varphi$ is a bijection.
	
	It remains to prove that $\varphi$ preserves adjacency and non-adjacency. For any $Q,Q^\prime \in \mathcal{Q}_k$ such that $Q \sim_{Q(k)} Q^\prime$, it is clear that $\varphi(Q) \cup \varphi(Q^\prime)$ does not contain any isolated edge, otherwise, $Q \cup Q^\prime$ would contain a $2$-cycle or an isolated vertex. If $P = \varphi(Q),P^\prime = \varphi(Q^\prime) \in P(k+1)$ are adjacent, then $P\cup P^\prime$ does not contain a $2$-cycle, so the removal of the vertex $2k+2$ cannot give rise to a graph with an isolated vertex or a graph with $2$-cycle. Consequently, $\varphi$ is an isomorphism.
\end{proof}

It is not hard to see that the isomorphism $\varphi$ also induces an isomorphism on the orbital graphs that are subgraphs of $Q(k)$ and $P(k+1)$.

An immediate corollary of the above lemma and Theorem~\ref{thm:pm} is the following.
\begin{thm}
	If $\mathcal{F} \subset \mathcal{Q}_k$ is intersecting, then $|\mathcal{F}| \leq (2k-1)!!$. Equality holds if and only if $\mathcal{F}$ is an orbit of size $(2k-1)!!$ of a conjugate of $\sym([2k,1])$ or a conjugate of $\sym([2k-1,2])$.
\end{thm}

We note that the partitions $[2k,1]$ and $[2k-1,2]$ are exactly the two resulting partitions obtained by applying the Branching Rule from $\sym(2k+1)$ to $\sym(2k)$ (see \cite{sagan2001symmetric} for details) on the partition $[2k,2]$.

Since $Q(k)$ is a union of orbital graphs of $\sym(2k+1)$ acting on $\mathcal{Q}_k$, we conclude that the orbit of $\sym([2k,1])$ of size $(2k-1)!!$ is a coclique of any orbital graph contained in $Q(k)$. Hence, Question~\ref{qst:main}~\eqref{qst:part1} is affirmative. 

For Question~\ref{qst:main}~\eqref{qst:part2}, we consider the orbital graph $X$ corresponding to $\mathcal{O}_{[2k+1]}$. Due to Question~\ref{qst:main}~\eqref{qst:part1} being affirmative, $-(2k-2)!!$ is the eigenvalue of $X$ afforded by $[2k,1]$-module. Let $Y$ be the subgraph of $P(k+1)$ which is the orbital  graph of $\sym(2k+2)$ acting on $\mathcal{P}_{k+1}$ corresponding to the partition $[2k+2]$. It is known that the smallest eigenvalue $-(2k-2)!!$ of $Y$ is afforded by the $[2k,2]$-module and is the least eigenvalue of $Y$. Therefore, $-(2k-2)!!$ is also the smallest eigenvalue of $X,$ since  $X$ and $Y$ are isomorphic.

\section{Conclusion and future work}\label{sect:conclusion}

In this paper, we proved in Theorem~\ref{thm:main2} that for any Gelfand pair $(G,H)$, where $G = \sym(n)$ and $H$ is listed in lines~1-6 of Table~\ref{table:intrans2}, there exists a graph in the corresponding orbital scheme for which Question~\ref{qst:main} is affirmative. We also gave an example where the answer to Question~\ref{qst:main} is negative.
 
We do not know if Question~\ref{qst:main} is true in general for all multiplicity-free subgroups $H$ of $\sym(n)$ such that the second largest in dominance ordering in $\Lambda(n,H)$ is $[n-1,1]$. Therefore, we pose the following problem. 

\begin{prob}
	Determine whether Question~\ref{qst:main} is true for all other Gelfand pairs $(\sym(n),H)$ in which $[n-1,1]$ is the second largest in $\Lambda(n,H)$.
\end{prob}

For orbital schemes arising from other multiplicity-free subgroups, we expect the case where the second largest in dominance ordering in $\Lambda(n,H)$ is not equal to $[n-1,1]$ to be more complicated than the case considered in this paper. We provide some computational results on the small multiplicity-free subgroups obtained from \verb*|Sagemath| \cite{sagemath} in Table~\ref{tab:table2}.

\begin{table}[t]
	\begin{center}
		\begin{tabular}{l ccclcc}
			Group & $n$ & index & rank& Second largest & Question~\ref{qst:main}\eqref{qst:part1} & Question~\ref{qst:main}\eqref{qst:part2} \\
			$AGL(1,5) \cap A_5$ & $5$ & $12$ & $4$&$[3,2]$ & Yes & Yes \\
			$AGL(1,5)$ & $5$ & $6$ & $2$ & $[2^2,1]$ & No & No \\
			$PSL(2,5)$ & $6$ & $12$ & $4$ &  $[3^2]$& Yes & Yes \\
			$PGL(2,5)$ & $6$ & $6$ & $2$ & $[2^3]$ & No & No \\
			$AGL(1,7)$ & $7$ & $120$ & $7$ & $[4,3]$ & No & No \\
			$PSL(3,2)$ & $7$ & $30$ & $4$ & $[4,3]$ & Yes & No \\
			$A\Gamma L(1,8)$ & $8$ & $240$ & $8$ & $[5,1^3]$ & No &  No\\
			$PGL(2,7)$ & $8$ & $120$ & $5$ & $[4,4]$ & No & No \\
			$AGL(3,2)$ & $8$ & $30$ & $4$ & $[4^2]$ & Yes & No \\
			$\sym(2) \wr \alt(3)$ & $6$ & $30$ &$5$ & $[4,2]$ & Yes & No \\
			$\sym(2) \wr \alt(4)$ & $6$ & $30$ &$5$ & $[6,2]$ & Yes & No \\
		\end{tabular}
		\caption{Answer to Question~\ref{qst:main} for small multiplicity-free subgroup. \label{tab:table2}}
	\end{center}
\end{table}

We also ask the following problem.
\begin{prob}
	Find an EKR type theorem for partially $2$-intersecting families of $U_k$. That is, determine the maximum cocliques of the folded Johnson graph $\overline{J}(2k,k)$.
\end{prob}


\noindent{\bf Acknowledgement.} \ 
The research of both authors are supported in part by the Ministry of Education, Science and Sport of Republic of Slovenia (University of Primorska Developmental funding pillar).
\vspace*{0.5cm}

\noindent {\bf Conflict of interest.} None.


\begin{thebibliography}{10}
	
	\bibitem{behajaina20213}
	A.~Behajaina, R.~Maleki, A.~T. Rasoamanana, and A.~S. Razafimahatratra.
	\newblock 3-setwise intersecting families of the symmetric group.
	\newblock {\em Discrete Mathematics}, {\bf 344}(8):112467, 2021.
	
	\bibitem{behajaina2023intersection}
	A.~Behajaina, R.~Maleki, and A.~S. Razafimahatratra.
	\newblock On the intersection density of the symmetric group acting on uniform
	subsets of small size.
	\newblock {\em Linear Algebra Appl.}, {\bf 664}:61--103, 2023.
	
	\bibitem{bose1952classification}
	R.~C. Bose and T.~Shimamoto.
	\newblock Classification and analysis of partially balanced incomplete block
	designs with two associate classes.
	\newblock {\em J. Amer. Statist. Assoc.}, {\bf 47}(258):151--184, 1952.
	
	\bibitem{brouwer1989distance}
	A.~Brouwer, A.~Cohen, and A.~Neumaier.
	\newblock {\em Distance-Regular Graphs}.
	\newblock Ergebnisse der Mathematik und ihrer Grenzgebiete. 3. Folge / A Series
	of Modern Surveys in Mathematics. Springer Berlin Heidelberg, 1989.
	
	\bibitem{ceccherini2008harmonic}
	T.~Ceccherini~Silberstein, F.~Scarabotti, and F.~Tolli.
	\newblock {\em Harmonic analysis on finite groups}, volume~{\bf 108}.
	\newblock Cambridge University Press, 2008.
	
	\bibitem{Frankl1977maximum}
	M.~Deza and P.~Frankl.
	\newblock On the maximum number of permutations with given maximal or minimal
	distance.
	\newblock {\em J. Combin. Theory Ser. A}, {\bf 22}(3):352--360, 1977.
	
	\bibitem{erdos1961intersection}
	P.~Erd\H{o}s, C.~Ko, and R.~Rado.
	\newblock Intersection theorems for systems of finite sets.
	\newblock {\em Q. J. Math.}, {\bf 12}(1):313--320, 1961.
	
	\bibitem{godsil2010multiplicity}
	C.~Godsil and K.~Meagher.
	\newblock Multiplicity-free permutation representations of the symmetric group.
	\newblock {\em Ann. of Comb.}, {\bf 13}(4):463--490, 2010.
	
	\bibitem{godsil2016algebraic}
	C.~Godsil and K.~Meagher.
	\newblock An algebraic proof of the {E}rd{\H{o}}s-{K}o-{R}ado theorem for
	intersecting families of perfect matchings.
	\newblock {\em Ars Math. Contemp.}, {\bf 12}(2):205--217, 2016.
	
	\bibitem{godsil2016erdos}
	C.~Godsil and K.~Meagher.
	\newblock {\em {E}rd\H{o}s-{K}o-{R}ado {T}heorems: {A}lgebraic {A}pproaches}.
	\newblock Cambridge University Press, 2016.
	
	\bibitem{haemers2021hoffman}
	W.~H. Haemers.
	\newblock Hoffman's ratio bound.
	\newblock {\em Linear Algebra Appl.}, {\bf 617}:215--219, 2021.
	
	\bibitem{higman1970coherent}
	D.~G. Higman.
	\newblock Coherent configurations i.
	\newblock {\em Rendiconti del Seminario Matematico della Universita di Padova},
	{\bf 44}:1--25, 1970.
	
	\bibitem{katona1972simple}
	G.~Katona.
	\newblock A simple proof of the {E}rd{\H{o}}s-{K}o-{R}ado theorem.
	\newblock {\em J. Combin. Theory Ser. B}, {\bf 13}(2):183--184, 1972.
	
	\bibitem{lindzey2017erdHos}
	N.~Lindzey.
	\newblock Erd{\H{o}}s--{K}o--{R}ado for perfect matchings.
	\newblock {\em European J. Combin.}, {\bf 65}:130--142, 2017.
	
	\bibitem{martin2009commutative}
	W.~J. Martin and H.~Tanaka.
	\newblock Commutative association schemes.
	\newblock {\em European J. Combin.}, {\bf 30}(6):1497--1525, 2009.
	
	\bibitem{sagan2001symmetric}
	B.~E. Sagan.
	\newblock {\em The Symmetric Group: Representations, Combinatorial Algorithms,
		and Symmetric Functions (Graduate Texts in Mathematics)}.
	\newblock New York: Springer, 2001.
	
	\bibitem{sagemath}
	{The Sage Developers}.
	\newblock {\em {S}ageMath, the {S}age {M}athematics {S}oftware {S}ystem
		({V}ersion 9.7)}, 2022.
	\newblock {\tt https://www.sagemath.org}.
	
	\bibitem{wilson1984exact}
	R.~M. Wilson.
	\newblock The exact bound in the {E}rd{\H{o}}s-{K}o-{R}ado theorem.
	\newblock {\em Combinatorica}, {\bf 4}(2-3):247--257, (1984).
	
\end{thebibliography}
 \end{document}